\documentclass[11pt,a4paper]{amsart}

\usepackage[a-3u]{pdfx}

\usepackage[T1]{fontenc}
\usepackage{lmodern,ae,aecompl}

\usepackage{amsmath,amsfonts,amsthm,amssymb}
\usepackage{mathtools}

\usepackage{enumitem,graphicx}

\usepackage[backend=biber,articlein=false,style=ext-numeric,sorting=nty,doi=false,isbn=false,url=false,eprint=false,giveninits=true]{biblatex}

\usepackage{hyperref}\hypersetup{unicode,breaklinks=true,colorlinks=true,allcolors=blue}
\usepackage{cleveref}

\DeclareMathOperator{\aconv}{aconv}
\DeclareMathOperator{\conv}{conv}
\DeclareMathOperator{\Lip}{Lip}
\DeclareMathOperator{\lip}{lip}

\allowdisplaybreaks

\usepackage{macros}

\begin{document}

\title{Lipschitz Free Spaces and Subsets of Finite-Dimensional Spaces}
\begin{abstract}
    We consider two questions on the geometry of Lipschitz free $p$-spaces $\mathcal F_p$, where $0<p\leq 1$, over subsets of finite-dimensional vector spaces. We solve an open problem and show that if $(\mathcal M, \rho)$ is an~infinite doubling metric space (e.g., an~infinite subset of an Euclidean space), then $\mathcal F_p (\mathcal M, \rho^\alpha)\simeq\ell_p$ for every $\alpha\in(0,1)$ and $0<p\leq 1$. An upper bound on the Banach-Mazur distance between the spaces $\mathcal F_p ([0, 1]^d, |\cdot|^\alpha)$ and $\ell_p$ is given. Moreover, we tackle a~question due to~\textcite{Albiac2022} and expound the role of $p$, $d$ for the Lipschitz constant of a~canonical, locally coordinatewise affine retraction from $(K, |\cdot|_1)$, where $K=\bigcup_{Q\in \mathcal R} Q$ is a~union of a~collection $\emptyset \neq \mathcal R \subseteq \{ Rw + R[0,1]^d: w\in\mathbb Z^d\}$ of cubes in $\mathbb R^d$ with side length $R>0$, into the Lipschitz free $p$-space $\mathcal F_p (V, |\cdot|_1)$ over their vertices.
\end{abstract}
\author{Jan B\'ima}

\email{jan.bima@mff.cuni.cz}
\address{Charles University, Faculty of Mathematics and Physics, Department of Mathematical Analysis, Sokolovsk\'a 83, 186 75 Prague 8, Czech Republic}
\address{MSD Czech Republic, Prague, Czech Republic}

\thanks{I wish to thank M. Cúth for his countless insights into my research efforts. I would also like to thank J. L. Ansorena for suggesting~\Cref{cexample:cexample_extension}. I acknowledge the~support of GAČR 23-04776S. The work was completed while the author was an employee at MSD Czech Republic, Prague, Czech Republic.}
\maketitle

\section{Introduction}\label{sec:intro}

Given a~pointed metric space $\mathcal M$, there exists a~Banach space $\mathcal F(\mathcal M)$, called the~\emph{Lipschitz free space over $\mathcal M$}, such that $\mathcal M$ embeds isometrically into $\mathcal F(\mathcal M)$ via a~map $\delta:\mathcal M \to \mathcal F(\mathcal M)$, and for every Banach space $Y$ and a~Lipschitz map $f:\mathcal M\to Y$ which vanishes at the origin, $f$ extends uniquely to a~linear operator $T_f : \mathcal F(\mathcal M)\to Y$ such that $\Lip f = \lVert T_f\rVert$.

Lipschitz free spaces have the distinguishing property that they relate the classical linear theory to the non-linear geometry of Banach spaces. This line of research goes back to the seminal paper by~\textcite{Godefroy2003}, who identified Lipschitz free spaces as a~natural class of objects to study the classical, deep problem whether two Lipschitz isomorphic Banach spaces are linearly isomorphic.

To give an~application of the theory, we note the authors were able to establish that whenever $X$ and $Y$ are Banach spaces and $X$ embeds into $Y$ isometrically, then there exists a~\emph{linear} isometric embedding of $X$ into $Y$. Similarly, they used the universal extension property of Lipschitz free spaces to show that a~\emph{bounded approximation property} of a~Banach space is preserved merely by Lipschitz isomorphisms. Let us remark that the study of approximation properties and the non-linear geometry of Banach spaces is an~ongoing topic (see~\cite{Albiac2022,Borel2012,Godefroy2020,Hajek2022,Kalton2012,Pernecka2015}).

In the context of the Lipschitz isomorphism problem,~\textcite{Albiac2009} later came with an~example of two separable $p$-Banach spaces, for each $0<p<1$, which are Lipschitz isomorphic but fail to be linearly isomorphic. As it turns out, the counterexample to the generalized variant of the Lipschitz isomorphism problem could be developed in the setting of generalized Lipschitz free spaces, coined the~\emph{Lipschitz free $p$-spaces}.

For each $0<p\leq 1$, the Lipschitz free $p$-space $\mathcal F_p (\mathcal M)$ over a~metric space $\mathcal M$ is a~$p$-Banach space into which $\mathcal M$ isometrically embeds, and such that for every $p$-Banach space $Y$ and a~Lipschitz map $f:\mathcal M\to Y$ which vanishes at the origin, $f$ extends uniquely to a~linear operator $T_f : \mathcal F_p(\mathcal M)\to Y$ with $\Lip f = \lVert T_f\rVert$. We note that a~thorough study of Lipschitz free $p$-space was recently initiated in~\cite{Albiac2020}.

The locally non-convex geometry of Lipschitz free $p$-spaces is rather challenging to grasp. To name an~evidence, we remark that for any subspace $\mathcal N$ of a~metric space $\mathcal M$, it is straightforward to show that $\mathcal F_1 (\mathcal N)$ embeds isometrically into $\mathcal F_1 (\mathcal M)$ via a~canonical linearization of the inclusion map $i: \mathcal N \to \mathcal M$. However, this is not the case for $p<1$, and it is even an~open question whether the inclusion in general is an~isomorphic embedding, see~\cites[Theorem~6.1 and~Question~6.2]{Albiac2020}, respectively.

A~distinctive feature of the $p<1$ theory is that a~duality argument is no longer at our disposal, and we instead have to proceed by a~direct geometrical construction in the Lipschitz free $p$-space itself. Moreover, a~strict concavity of a~$p$-norm for $p<1$ typically introduces a~dimensionality factor into the proof work; typically, this would render many of the techniques developed within the vast literature dedicated to approximation properties of Lipschitz free spaces hardly adaptable.

Here we consider two open questions on the structure of Lipschitz free $p$-spaces over subsets of finite-dimensional normed spaces. In particular, we expound the extent to which selected results from the classical $p=1$ theory generalize to the $0<p\leq 1$ scale.

\begin{introthm}[{cf.~\Cref{thm:isomorphism_p}}]\label{introthm:1}
    Let $(\mathcal M, \rho)$ be an~infinite doubling metric space (e.g., an~infinite subset of an~Euclidean space) and $0 < \alpha < 1$, $0<p\leq 1$. Then $\mathcal F_p (\mathcal M, \rho^\alpha)$ is isomorphic to the space~$\ell_p$.
\end{introthm}

A~classical result in the theory of Lipschitz free spaces states that if $|\cdot|$ is a~norm on $\mathbb R^d$ and $\mathcal M$ is an~infinite bounded subset of $\mathbb R^d$ endowed with the~\emph{Hölder distorted metric} $|\cdot|^\alpha$, where $0<\alpha<1$, then $\mathcal F_1 (\mathcal M, |\cdot|^\alpha)\simeq \ell_1$.

The standard approach (consider e.g.~\cite{Weaver2018}) is to identify an~isometric predual of $\mathcal F_1 (\mathcal M, |\cdot|^\alpha)$ as the subspace $\lip_0 (\mathcal M, |\cdot|^\alpha)$ consisting of~\emph{little Lipschitz} functions in the Lipschitz dual $\Lip_0 (\mathcal M, |\cdot|^\alpha)\simeq\mathcal F_1^* (\mathcal M, |\cdot|^\alpha)$. The proof then proceeds by constructing an~isomorphism between $\lip_0 (\mathcal M, |\cdot|^\alpha)$ and the space $c_0$; this is an~earlier result which traces back to, e.g.,~\cite{Bonic1969}. 

More recently, the result was generalized in~\cite{Albiac2021sums} to infinite subsets of $\mathbb R^d$. In particular, an~observation was made showing that for any $0<p\leq 1$, if $\mathcal F_p ([0,1]^d, |\cdot|^\alpha)$ is isomorphic to the space $\ell_p$, then $\mathcal F_p (\mathcal M, |\cdot|^\alpha)\simeq\ell_p$ for any infinite subset $\mathcal M$ of $\mathbb R^d$ (and, by virtue of Assouad's embedding theorem, to Hölder distortions of infinite doubling metric spaces). The authors claimed that the ideas from the standard $p=1$ argument adopt to yield the isomorphism $\mathcal F_p ([0,1]^d, |\cdot|^\alpha)\simeq\ell_p$ for $d=1$ and $0<p\leq 1$; however, for $d\geq 2$ the available techniques turned insufficient and the problem remained open, see~\cite[Question~6.8]{Albiac2021sums}.

Here we tackle the multidimensional structure of $\mathcal F_p ([0,1]^d, |\cdot|^\alpha)$, and unlike the~standard proof for $p=1$, we set up an~explicit linear bijection between $p$-norming sets in the respective spaces. As it turns out, the~basis shares the form with the Schauder basis of $\mathcal F_p ([0,1]^d)$, see~\cite[Theorem~3.8]{Albiac2022}.

It is also interesting to note that our approach gives an~estimate on the Banach-Mazur distance between $\mathcal F_p ([0,1]^d, |\cdot|^\alpha)$ and $\ell_p$ whenever $|\cdot|$ is identified as the $\ell_1$ norm, which is a~new detail even for the case $p=1$. An~interested reader may want to compare the upper bound of $\left(4 d^{2-\alpha} c(\alpha) \right)^d$ with a~lower bound of $c'(\alpha)d^{\alpha(1-\alpha)}(\log (2n))^{-\alpha/2}$ whenever $\alpha\in [1/2, 1)$ and $c'(\alpha)d^{\alpha/2}(\log (2n))^{-\alpha/2}$ otherwise, where $p=1$ and $c(\alpha)$, $c'(\alpha)$ are universal constants, see~\cite[Proposition~8.6]{Kalton2004}.

As an introduction to the proof of~\Cref{introthm:1}, it will be instructive to better investigate a~canonical, locally coordinatewise affine retraction from $(K, |\cdot|_1)$, where $K=\bigcup_{Q\in \mathcal R} Q$ is a~union of a~collection $\emptyset \neq \mathcal R \subseteq \{ Rw + R[0,1]^d: w\in\mathbb Z^d\}$ of cubes in $\mathbb R^d$ with side length $R>0$, into the Lipschitz free $p$-space $\mathcal F_p (V, |\cdot|_1)$ over their vertices.

From~\cite{Lancien2013} we know that for $p=1$, the retraction is Lipschitz continuous with the Lipschitz constant equal to one. More recently, it was established in~\cite[Theorem~5.1]{Albiac2022} that the retraction is Lipschitz continuous for any $0<p\leq 1$; however, the role of $p$, $d$ for the estimate of the Lipschitz constant was unclear and the method led to an~suboptimal estimate even for the classical $p=1$ case.

Here we present an~alternative approach which generically refines the estimate from~\cite{Albiac2022}, and we apply a~double counting argument to derive a~lower bound on the Lipschitz constant of the retraction. That is, we obtain the following result, which answers~\cite[Question~4.6]{Albiac2022} in the negative.

\begin{introthm}[{cf.~\Cref{thm:lattice_embedding}}]\label{introthm:2}
    There is a~unique map $r_{K, V}:(K, |\cdot|_1)\to\mathcal F_p (V, |\cdot|_1)$ such that $r_{K, V}(v) = \delta_V (v)$, where $v\in V$, and $r_{K, V}$ is coordinatewise affine on each of the cubes in $\mathcal R$. If we denote $C(p, n)= n^{1/p -1}$, where $n\in\mathbb N$, then \[ C(p, 2^{d-1}) \leq\Lip r_{K, V} \leq C(p, 2^{d-1})C(p, d)C(p, 3) \text.\]
\end{introthm}

\smallskip
The article is organized as follows. In~\Cref{sec:preliminaries}, we recall the notion of a~$p$-Banach space and include several foundational properties of Lipschitz free $p$-spaces. We also introduce the canonical, locally coordinatewise affine retraction in a~cube. \Cref{sec:retraction} is devoted to the proof of~\Cref{introthm:2}. In~\Cref{sec:isomorphism}, we develop a~series of results leading up to the proof of~\Cref{introthm:1} for the particular case $\mathcal M=[0,1]^d$, and then we deduce the general conclusion for Hölder distortions of infinite doubling metric spaces.

\section{Preliminaries}\label{sec:preliminaries}

\subsection{\texorpdfstring{$p$}{p}-Normed Spaces}

\begin{defn}
    Let $X$ be a~vector space. We say that a~map $\lVert \cdot \rVert : X\to [0, \infty)$ is a~\emph{quasi-norm} on $X$ if there exists $\kappa \geq 1$ such that
    \begin{enumerate}[label=(\roman*), ref=\roman*]
        \item $\lVert x \rVert > 0$ for any $x\neq 0$,
        \item\label{it:quasi_norm_homogeneity} $\lVert \alpha x \rVert = |\alpha|\lVert x \rVert$ for any scalar $\alpha$ and $x\in X$,
        \item\label{it:quasi_norm} $\lVert x+y\rVert \leq \kappa(\lVert x\rVert + \lVert y\rVert)$ for any $x,\,y\in X$.
    \end{enumerate}
    
    We then call $(X, \lVert\cdot\rVert)$ a~\emph{quasi-normed space}. 
    
    Replacing~\cref{it:quasi_norm} with the assumption that for some $0<p\leq 1$,
    \begin{enumerate}[start=3,label=(\roman*'), ref=\roman*']
        \item\label{it:p_norm} $\lVert x+y\rVert^p \leq \lVert x\rVert^p + \lVert y\rVert^p$ for any $x,\,y\in X$,
    \end{enumerate}
    we obtain the notion of a~\emph{$p$-norm} and a~\emph{$p$-normed space}. If moreover $X$ is complete with respect to the metric $d(x,y)=\lVert x-y\rVert^p$, where $x,\,y\in X$, we say $(X, \lVert\cdot\rVert)$ is a~\emph{$p$-Banach space}.
\end{defn}

\begin{defn}
    For $0<p\leq 1$, we say that a~subset $Z$ of a~vector space $X$ is \emph{absolutely $p$-convex} if for any $x,\, y\in Z$ and scalars $\alpha,\,\beta$, where $|\alpha|^p+|\beta|^p\leq 1$, we have $\alpha x + \beta y\in Z$. 
    
    The smallest absolutely $p$-convex set containing $Z$ is denoted by $\aconv_p Z$.
\end{defn}

We shall write $B_{X} = \{ x\in X : \lVert x \rVert\leq 1\}$ for a~\emph{unit ball} of a~quasi-normed space $(X, \lVert\cdot\rVert)$.

\begin{defn}\label{def:p_norming}
    For $0<p\leq 1$, we say that a~subset $Z$ of a~quasi-normed space $X$ is \emph{$p$-norming with constants $\alpha, \,\beta > 0$} whenever \[ \alpha\, \overline\aconv_p Z \subseteq B_X \subseteq \beta\,\overline \aconv_p Z \text.\]

    If $\alpha=\beta=1$, we say $Z$ is \emph{isometrically $p$-norming}.
\end{defn}

The following fact is an~easy linear variant of an extension theorem for Lipschitz continuous maps.

\begin{lemma}\label{fact:extension}
    Let $0<p\leq 1$. Assume that $Y_1$ and $Y_2$ are $p$-norming in $p$-Banach spaces $X_1$ and $X_2$, respectively, and that $\aconv_p Y_1$ and $\aconv_p Y_2$ contain neighborhoods of zero in $\operatorname{span} Y_1$ and $\operatorname{span} Y_2$, respectively. That is, we have $\aconv_p Y_i \supseteq c_i B_{X_i} \cap \operatorname{span} Y_i$ for some $c_i > 0$, for each $i\in \{1, 2\}$.
    
    If $T$ is a~one-to-one linear map from $\operatorname{span} Y_1$ into $X_2$ such that $T(Y_1)=Y_2$, then $T$ extends to an~onto isomorphism $\widetilde T : X_1 \to X_2$.

    Quantitatively, if $Y_1$, $Y_2$ are $p$-norming in $X_1$, $X_2$ with constants $\alpha, \,\beta$ and $\alpha', \,\beta'$, respectively, then $\lVert \widetilde T \rVert \leq \beta/\alpha'$ and $\lVert {\widetilde T}^{-1} \rVert \leq \beta'/\alpha$.
\end{lemma}

We remark that~\cite[Lemma~2.6]{Albiac2020} states a~stronger, alleged variant of the extension result, leaving out the assumption that $\aconv_p Y_i$ contains a~neighborhood of zero in $\operatorname{span} Y_i$, for each $i\in \{1,2\}$. This claim, however, is not true. Nevertheless, it turns out that the~assumptions of~\Cref{fact:extension} are satisfied in the applications of~\cite[Lemma~2.6]{Albiac2020} in~\cite{Albiac2020}; hence, the derived results remain valid.

The following counterexample to~\cite[Lemma~2.6]{Albiac2020} was suggested by J. L. Ansorena.

\begin{cexample}\label{cexample:cexample_extension}
Let $X$ be a~$p$-Banach space, where $0<p\leq 1$, and let $M$ be a~closed subspace of $X$. Let $T:X\to X/M$ be the quotient map.

Pick a~dense subspace $V$ of $X$ such that $V\cap M = \{ 0 \}$, and set $K = B_X \cap V$. It is easy to see that $K$ is absolutely $p$-convex as well as isometrically $p$-norming in $X$. If $U_X$ denotes the open unit neighborhood of zero in $X$, we verify that \[ B_{X/M} = \bigcap_{t>1} tT(U_X) = \bigcap_{t>1} tT(B_X) = \bigcap_{t>1} tT(\overline{K}) \subseteq \bigcap_{t>1} t\overline{T(K)} = \overline{T(K)} \text. \]

For instance, for any $0<p\leq 1$ we may take $X=\ell_p$, $M=\operatorname{span} \{ e_1 \}$, and $V = \operatorname{span} \left\{\{ e_n: n\geq 2\} \cup \{e_1 + \sum_{n=2}^\infty 2^{-n} e_n \}\right\}$.

If we denote $X_1=X$, $X_2=X/M$, and $Y_1 = K$, $Y_2 = T(K)$, then $Y_1$ and $Y_2$ are absolutely $p$-norming in $X_1$ and $X_2$, respectively, and $T$ is a~linear bijection from $\operatorname{span} Y_1$ onto $\operatorname{span} Y_2$. However, $T$ does not extend to an~isomorphism from $X_1$ into $X_2$.
\end{cexample}

\subsubsection*{A~Particular Class of Coefficients}

We introduce a~coefficient $C(p, n)$ which has the role of $\kappa$ in~\cref{it:quasi_norm} for sums of $n$ elements, i.e., if $(X, \lVert\cdot\rVert)$ is a~quasi-normed space and $0<p\leq 1$, $n\in \mathbb N$ are given, then $\lVert \sum_{i=1}^n x_i\rVert \leq C(p, n)\sum_{i=1}^n \lVert x_i\rVert$ for any $x_1, \ldots, x_n \in X$.

\begin{defn} \label{defn:const}
    For any $n \in \mathbb N$ and $0<p \leq 1$, let us denote \[C(p, n) = \sup \left\{ \left(\sum_{i=1}^n w_i^p\right)^{1/p} : w_i \geq 0 \text{ for } i\in \{1, \ldots, n\},\:\sum_{i=1}^n w_i \leq 1\right\} \text.\]
\end{defn}

Note that $\left(\sum_{i=1}^n |w_i|^p\right)^{1/p} \leq C(p, n)|w|_1$ for any $n\in\mathbb N$, $0<p \leq 1$, and $w=(w_i)_{i=1}^n \in \mathbb R^n$. 

An~explicit formula for $C(p, n)$ follows easily from Hölder's inequality.

\begin{fact} \label{lemma:const_}
    Let $n \in \mathbb N$, $0<p \leq 1$. It holds that $C(p, n) = n^{1/p -1}$.
\end{fact}

\subsection{Lipschitz Free~\texorpdfstring{$p$}{p}-Spaces}

If $(\mathcal M, \rho)$ is a~pointed metric space with $0_{\mathcal M}$ as its base point, we consider $\delta : \mathcal M \to \Lip_0 (\mathcal M)^*$ which maps $x\in\mathcal M$ to the canonical evaluation functional $\delta(x)\in \Lip_0 (\mathcal M)^*$, i.e., $\langle \delta(x), f \rangle = f(x)$ for each $f \in \Lip_0 (\mathcal M)$.

We recall that the~\emph{Lipschitz free space $\mathcal F (\mathcal M)$} over $\mathcal M$ can be identified as the closed span of $\delta(\mathcal M)$ in $\Lip_0 (\mathcal M)^*$, \[ \mathcal F (\mathcal M) = \overline{\operatorname{span}} \{ \delta(x) : x\in \mathcal M \} \text, \] and $\mathcal F(\mathcal M)^*$ is linearly isometric to $\Lip_0 (\mathcal M)$. In fact, if $\operatorname{Mol} (\mathcal M)$ denotes the set of~\emph{elementary molecules} in $\mathcal F(\mathcal M)$, \[ \operatorname{Mol} (\mathcal M) = \left\{ \frac{\delta(x)-\delta(y)}{\rho(x, y)} : x,\,y\in \mathcal M,\, x\neq y \right\}\text, \] it follows by virtue of Hahn-Banach Theorem that $B_{\mathcal F(\mathcal M)} = \overline\conv\operatorname{Mol} (\mathcal M)$.

\smallskip
Let $0<p\leq 1$ and denote $\mathcal P (\mathcal M) = \operatorname{span} \{ \delta(x) : x\in\mathcal M\}$. Drawing from the outlined construction, we set for each $m\in \mathcal P(\mathcal M)$ \begin{equation*} \left\Vert m \right\Vert = \inf \left( \sum_{i=1}^n |a_i|^p \right)^{1/p} \text,\end{equation*} the infimum being taken over all $n\in\mathbb N_0$ and $\mu_i\in\operatorname{Mol} (\mathcal M)$, $a_i\in\mathbb R$, for each $i\in\{1, \ldots, n\}$, such that $m = \sum_{i=1}^n a_i\mu_i$.

It turns out that $(\mathcal P(\mathcal M), \lVert\cdot\rVert)$ is a~$p$-normed space and $\delta$ is an~isometric embedding. A~completion process then results in the Lipschitz free $p$-space $\mathcal F_p(\mathcal M)$.

\begin{thm}[{cf.~\cite[Theorem~4.5]{Albiac2020}}]
    Let $(\mathcal M, \rho)$ be a~pointed metric space. Given $0<p\leq 1$, there exists a~$p$-Banach space $(\mathcal F_p (\mathcal M), \lVert\cdot\rVert)$, called the~\emph{Lipschitz free $p$-space over $\mathcal M$}, and a~map $\delta : \mathcal M \to \mathcal F_p (\mathcal M)$ such that
    \begin{props}[label=(\roman*), ref=\roman*]
        \item\label{it:F_p_isometry} $\delta$ is an~isometric embedding with $\delta(0_{\mathcal M})=0_{\mathcal F_p(\mathcal M)}$,
        \item\label{it:F_p_dense} $\mathcal F_p (\mathcal M) = \overline{\operatorname{span}} \{ \delta(x) : x\in \mathcal M \}$,
        \item\label{it:F_p_extension} if $(Y, \lVert\cdot\rVert_Y)$ is a~$p$-Banach space, then $B(\mathcal F_p(\mathcal M), Y)$ is linearly isometric to $\Lip_0 (\mathcal M, Y)$ via the map $f^* \mapsto f^* \circ \delta$ for each $f^*\in B(\mathcal F_p(\mathcal M), Y)$.
    \end{props}
\end{thm}

\begin{fact}[{cf.~\cite[Corollary~4.11]{Albiac2020}}]\label{fact:cf:isometrically_norming}
    Let $(\mathcal M, \rho)$ be a~pointed metric space. For each $0<p\leq 1$, the set $\operatorname{Mol} (\mathcal M)$ is isometrically $p$-norming in $\mathcal F_p (\mathcal M)$. That is, $B_{\mathcal F_p (\mathcal M)} = \overline\aconv_p\operatorname{Mol} (\mathcal M)$, and for each $m\in \mathcal P(\mathcal M)$ we have \begin{equation*} \left\Vert m \right\Vert = \inf \left( \sum_{i=1}^n |a_i|^p \right)^{1/p} \text,\end{equation*} the infimum being taken over all $n\in\mathbb N_0$ and $\mu_i\in\operatorname{Mol} (\mathcal M)$, $a_i\in\mathbb R$, where $i\in\{1, \ldots, n\}$, such that $m = \sum_{i=1}^n a_i\mu_i$.
\end{fact}

We show that for any dense subset $\mathcal N$ of $\mathcal M$ and any $m\in\mathcal P (\mathcal N)$, the above formula is still valid if we consider decompositions of $m$ merely into molecules over $\mathcal N$.

\begin{lemma}\label{lemma:isometrically_norming_subset}
    Let $\mathcal N$ be a~dense subset of a~pointed metric space $(\mathcal M, \rho)$. Then for each $0<p\leq 1$ and $m\in\mathcal P(\mathcal N)$, we have $\left\Vert m \right\Vert_{\mathcal F_p (\mathcal M)} = \inf \left( \sum_{i=1}^n |a_i|^p \right)^{1/p}$, the infimum being taken over all $n\in\mathbb N_0$ and $\mu_i\in\operatorname{Mol} (\mathcal N)$, $a_i\in\mathbb R$, where $i\in\{1, \ldots, n\}$, such that $m = \sum_{i=1}^n a_i\mu_i$.

    In particular, $\aconv_p \operatorname{Mol} (\mathcal N)$ contains the~open unit neighborhood of zero in $\mathcal P (\mathcal N)$, with respect to the ambient space $\mathcal F_p (\mathcal M)$. That is, we have $\{ m \in \mathcal P (\mathcal N): \lVert m \rVert_{\mathcal F_p (\mathcal M)} < 1 \} \subseteq \aconv_p \operatorname{Mol} (\mathcal N)$.
\end{lemma}

\begin{proof}
    Let $m\in \mathcal P (\mathcal N)$ and pick $\epsilon > 0$. It follows from~\Cref{fact:cf:isometrically_norming} that there exist $n\in\mathbb N_0$ and $\mu_i\in\operatorname{Mol} (\mathcal M)$, $a_i\in\mathbb R$, where $i\in\{1, \ldots, n\}$, for which $m = \sum_{i=1}^n a_i\mu_i$, and $\left( \sum_{i=1}^n |a_i|^p \right)^{1/p} < (1+\epsilon)\lVert m \rVert_{\mathcal F_p (\mathcal M)}$. We further consider $u_i,\,v_i \in \mathcal M$, $u_i\neq v_i$, such that $\mu_i = \frac{\delta(u_i)-\delta(v_i)}{|u_i-v_i|^\alpha}$, where $i\in\{1, \ldots, n\}$.
    
    Let us denote $\mathcal B = \{u_i\}_{i\in\{1, \ldots, n\}} \cup \{ v_i\}_{i\in\{1, \ldots, n\}}$. By density of $\mathcal N$ in $\mathcal M$, it is easy to construct a~mapping $r: \mathcal B \mapsto \mathcal N$ such that $r(b)=b$ for any $b\in \mathcal N$, and $\frac{|r(u_i)-r(v_i)|^{\alpha}}{|u_i-v_i|^{\alpha}}<1+\epsilon$ but $r(u_i)\neq r(v_i)$, for any $i\in\{1,\ldots, n\}$. 

    We consider the unique linear mapping $r':\mathcal P (\mathcal B) \mapsto \mathcal P (\mathcal N)$ which satisfies that $r'(\delta(b)) = \delta(r(b))$ for each $b\in \mathcal B$. It is easy to see that $r'(m) = m$, as $r'$ agrees with the identity on $\mathcal P (\mathcal N)$.

    Let us rewrite $m=r'(m)=\sum_{i=1}^n a_i r'(\mu_i)$, where  
    \begin{equation*}
        \begin{split}
            \sum_{i=1}^n a_i r'(\mu_i)&=\sum_{i=1}^n a_i \frac{r'(\delta(u_i))-r'(\delta(v_i))}{|u_i-v_i|^\alpha}\\
            &= \sum_{i=1}^n a_i \frac{|r(u_i)-r(v_i)|^\alpha}{|u_i-v_i|^\alpha}\frac{\delta(r(u_i))-\delta(r(v_i))}{|r(u_i)-(v_i)|^\alpha} \text.
        \end{split}
    \end{equation*}
    
    For each $i\in\{ 1, \ldots, n\}$, we denote $\mu'_i = \frac{\delta(r(u_i))-\delta(r(v_i))}{|r(u_i)-(v_i)|^\alpha} \in \operatorname{Mol} (\mathcal N)$ and set $a'_i = a_i \frac{|r(u_i)-r(v_i)|^\alpha}{|u_i-v_i|^\alpha}$. It follows from the construction that $m = \sum_{i=1}^n a'_i\mu'_i$ and $\left( \sum_{i=1}^n |a'_i|^p \right)^{1/p} < (1+\epsilon)\left( \sum_{i=1}^n |a_i|^p \right)^{1/p}<(1+\epsilon)^2\lVert m \rVert_{\mathcal F_p (\mathcal M)}$. We take $\epsilon \to 0$, and the first claim follows.

    Note that, in particular, the already proven part shows that $\{ m \in \mathcal P (\mathcal N): \lVert m \rVert_{\mathcal F_p (\mathcal M)} < 1 \} \subseteq \aconv_p \operatorname{Mol} (\mathcal N)$. This establishes the second claim.   
\end{proof}

We note that any Lipschitz map between pointed metric spaces which vanishes at the base point has an~extension to a~bounded operator between the respective Lipschitz free $p$-spaces, for every $0<p\leq 1$.

\begin{fact}[{cf.~\cite[Lemma~4.8]{Albiac2020}}]\label{fact:F_p_linearization}
    Let $\mathcal M$, $\mathcal N$ be pointed metric spaces. For every $0<p\leq 1$, there is a~linear isometry $L: \Lip_0 (\mathcal M, \mathcal N) \to B (\mathcal F_p(\mathcal M), \mathcal F_p(\mathcal N))$, called the~\emph{canonical linearization operator}, such that $\delta_{\mathcal N} \circ f = L(f) \circ \delta_{\mathcal M}$, for every $f\in \Lip_0 (\mathcal M, \mathcal N)$.
\end{fact}

\subsection{A~Projective Construction in~\texorpdfstring{$[0, 1]^d$}{[0, 1] d}}\label{subsec:proj_construction}

We overview a~canonical, locally coordinatewise affine projective construction in $[0, 1]^d$; hereby we set ground for basis expansions in $\mathcal F_p ([0,1]^d, |\cdot|^\alpha)$ considered within~\Cref{sec:isomorphism}. Let us remark this particular choice has in various contexts appeared thorough, e.g.,~\cite{Albiac2022,Weaver2018,Lancien2013}. We adopt the notation from~\cite[Section~3]{Albiac2022}. 

\medskip
Let $d\in\mathbb N$, $R>0$. If $w=(w_i)_{i=1}^d \in\mathbb Z^d$, we set $V^d_{w, R} = Rw+R\{0, 1\}^d$ and define a~cube $Q^d_{w, R}$ as the convex hull of the set $V^d_{w, R}$, i.e., $Q^d_{w, R} = Rw+R\left[0, 1\right]^d$. We denote $\mathcal{Q}^d_{R} = \left\{Q^d_{w, R}: w\in\mathbb{Z}^d \right\}$ and $\mathcal V^d_{R} = \left\{V^d_{w, R}: w\in\mathbb Z^d \right\}$. Let us further introduce a~map $\mathcal V : \mathcal{Q}^d_{R} \to \mathcal V^d_{R}$ by $\mathcal V(Q^d_{w, R})= V^d_{w, R}$, where $w\in\mathbb Z^d$.

\smallskip
We define projective coefficients from $\mathbb R^d$ to the vertex set $\mathcal V^d_{R}$. For any $x\in [0, 1]$ and $w\in\mathbb Z$, we first put
\begin{equation*}
    x^{(w)} = 
        \begin{cases} 
            x & \text{if } w=1, \\
            1-x & \text{if } w=0, \\
            0 & \text{if } w\notin\{0, 1\}, \\
        \end{cases}
\end{equation*}
and write, whenever $x=(x_i)_{i=1}^d \in [0, 1]^d$ and $w=(w_i)_{i=1}^d \in\mathbb Z^d$, \[x^{(w)} = \prod_{i=1}^d x_i^{(w_i)} \text. \]

We find this construction to admit a~lift to $\mathbb R^d$ in $x$.

\begin{lemma}[{cf.~\cite[Lemma~3.1]{Albiac2022}}]\label{lemma:cf:coeff_Lambda}
    Let $d\in \mathbb N$, $R>0$. There exists a~map \[\Lambda_R^d: \bigcup \mathcal V^d_R \times \mathbb R^d \to [0, 1]\] such that $\Lambda_R^d(Ru, Rw+Rx)=x^{(u-w)}$, for every $x\in[0, 1]^d$ and $u,\,w\in\mathbb Z^d$.
\end{lemma}

Moreover, we list properties of $\Lambda^d_R$ which we shall refer to in the sequel.

\begin{lemma}[{cf.~\cite[Lemma~3.4]{Albiac2022}}]\label{lemma:cf:coeff_Lambda_properties}
    Let $d\in \mathbb N$, $R>0$. It holds that
    \begin{enumerate}[label=(\roman*), ref=\roman*]
        \item\label{it:Lambda_1} $\Lambda^d_R (v, x) =0$ whenever $x\in Q \in\mathcal{Q}^d_{R}$ and $v\notin \mathcal V(Q)$,
        \item\label{it:Lambda_2} $\Lambda^d_R (v, u) = \delta_{v, u}$ for any $u,\,v\in \mathcal V^d_R$,
        \item\label{it:Lambda_3} $\sum_{v\in\mathcal V^d_R} \Lambda^d_R (v, x)=1$ for any $x\in\mathbb R^d$,
        \item\label{it:Lambda_4} $\Lambda^d_R (v, x) = \prod_{i=1}^d \Lambda^1_R (v_i, x_i)$ for any $x=(x_i)_{i=1}^d\in\mathbb R^d$, $v=(v_i)_{i=1}^d\in\mathcal V^d_R$.
    \end{enumerate}
\end{lemma}

\section{A~Retraction in~\texorpdfstring{$\mathcal F_p (\mathbb [0, 1]^d, |\cdot|_1)$}{Fp ([0, 1]d, ||1)}}\label{sec:retraction}

Drawing on the projective construction from~\Cref{subsec:proj_construction}, we consider a~retraction from $(K, |\cdot|_1)$, where $K = \bigcup_{Q\in\mathcal R} Q$ is a~union of a~collection $\mathcal R$ of regularly spaced cubes in $\mathbb R^d$ with equal side length, into the Lipschitz free $p$-space $\mathcal F_p (V, |\cdot|_1)$ over their vertices. We provide bounds on the Lipschitz constant thereof, and hereby we analyze locally coordinatewise affine extensions of Lipschitz maps from a~vertex set ranging into $p$-Banach spaces.

\medskip
We pick $d\in\mathbb N$ and endow $\mathbb R^d$ with the $\ell_1$ norm in the sequel.

Adopting the notation of~\Cref{subsec:proj_construction}, let $\mathcal R \subseteq \mathcal Q^d_R$, $K = \bigcup_{Q\in\mathcal R} Q$, $V = \bigcup_{Q\in \mathcal Q} \mathcal V(Q)$, and fix a point of $V$ as the base point of both $K$, $V$. As a~consequence of~\Cref{lemma:cf:coeff_Lambda_properties}~\labelcref{it:Lambda_1}, for any $x\in K$ the coefficients $\Lambda^d_R (\cdot, x)$ are finitely supported, and hence we are justified to introduce $r=r_{K, V}:K\to\mathcal F_p(V)$ as \begin{equation}\label{eq:def_r_} r(x) = \sum_{v\in V} \Lambda^d_R(v,x)\delta_V(v), \quad x\in K.\end{equation}

We can easily see that $r_{K,V}=\delta_V$ on $V$. Moreover, it was established in~\cite[Theorem~3.5]{Albiac2022} that, assuming the $\ell_\infty$ norm on $\mathbb R^d$, the map $r_{K, V}$ is Lipschitz with an upper bound depending on both $p$ and $d$. Applying their method to the $p=1$ and $\ell_1$ norm case, the obtained estimate still retained a~term depending on $d$, thus contrasting a~positive result due to~\cite{Lancien2013} which shows the Lipschitz constant to equal 1.

\begin{question}[{\cite[Question~4.6]{Albiac2022}}]
    Is there a~constant $C$ depending on $p$ but not on $d$, $K$, or $V$, such that $\Lip(r_{K, V})\leq C$?
\end{question}

We refine the estimate of~\cite[Theorem~3.5]{Albiac2022} but answer the above question in negative.

\begin{thm}\label{thm:lattice_embedding}
    Let $d\in\mathbb N$, $R>0$, $\mathcal R \subseteq \mathcal Q^d_R$, where $\mathcal R \neq \emptyset$, $K = \bigcup_{Q\in\mathcal R} Q$, $V = \bigcup_{Q\in \mathcal R} \mathcal V(Q)$, and consider $V$ as a~pointed metric space with the subspace $\ell_1$ metric. We let $0<p\leq 1$ and $r=r_{K, V}:K\to\mathcal F_p(V)$ be as above.
    
    For any $x,\, y \in K$ we have \[\left\Vert r(x)-r(y) \right\Vert_p \leq C(p, 2^{d-1})C(p, d)C(p, 3)|x-y|_1 \text.\] Conversely, there exist $x,\, y \in K$ such that \[ C(p, 2^{d-1})|x-y|_1 \leq \left\Vert r(x)-r(y) \right\Vert_p \text.\]
\end{thm}

We develop a~series of preliminary results which outline properties of the present construction under dilation and translation. 

\begin{lemma} \label{lemma:isometry_of_translation}
    Let $V \subseteq \mathbb Z^d$, $0_V \in V$, $R>0$, and $x\in\mathbb Z$. Denote $V^\prime = RV+Rx$, $0_{V'}=R0_V+Rx$, and consider $0_V$, $0_{V'}$ as the base points of $V$, $V^\prime$, respectively.
    
    For any $0<p\leq 1$, the map $\tau^\prime: \delta_{V}(V) \subseteq \mathcal F_p (V) \to \mathcal F_p (V^\prime)$, $\delta_{V}(v) \mapsto \delta_{V^\prime}(Rv+Rx)$, where $v\in V$, extends to an~isomorphism $\tau$ of $\mathcal F_p (V)$ and $\mathcal F_p (V^\prime)$, such that $\left\Vert\tau(x)\right\Vert_{\mathcal F_p (V^\prime)} = R\left\Vert x \right\Vert_{\mathcal F_p (V)}$ for any $x\in \mathcal F_p (V)$.
\end{lemma}

\begin{proof}
    We note the map $\sigma : V \to V^\prime$, $v \mapsto Rv+Rx$, where $v\in V$, is a~bi-Lipschitz bijection of $V$, $V^\prime$, such that $|\sigma(v)-\sigma(u)|_1=R|v-u|_1$, where $v,\,u\in V$.
    
    By~\Cref{fact:F_p_linearization} there exists an~isomorphism $\tau$ of spaces $\mathcal F_p (V)$ and $\mathcal F_p (V^\prime)$ satisfying $\tau \circ \delta_V = \delta_{V^\prime} \circ \sigma$ and $\left\Vert\tau(x)\right\Vert_{\mathcal F_p (V^\prime)} = R\left\Vert x \right\Vert_{\mathcal F_p (V)}$, where $x\in \mathcal F_p (V)$. Since $\tau\restriction_{\delta_{V}(V)} = \tau^\prime$ by the construction, the conlusion follows.
\end{proof}

\begin{lemma} \label{lemma:translation}
    Let $d\in\mathbb N$, $R>0$, $\mathcal R \subseteq \mathcal Q^d_R$, $K = \bigcup_{Q\in\mathcal R} Q$, and $V = \bigcup_{Q\in \mathcal R} \mathcal V(Q)$. Whenever $x,\,y \in K$ satisfy the condition $y-x\in R\mathbb Z^d$, we have that
    \begin{enumerate}[label=(\roman*), ref=\roman*]
        \item\label{it:translation_1} $\Lambda_R^d (v, x)=\Lambda_R^d (v+(y-x), y)$, where $v\in R\mathbb Z^d$,
        \item $v+(y-x)\in V$ for any $v\in V$, $\Lambda^d_R(v, x)\neq 0$,
        \item $r(y) = \sum_{v\in V,\;\Lambda^d_R(v, x)\neq 0} \Lambda^d_R(v, x)\delta_V(v+(y-x))$.
    \end{enumerate}
\end{lemma}

\begin{proof}
    Pick $w,\,u\in\mathbb Z^d$ such that $x\in Q^d_{w, R} \in \mathcal R$ and $y\in Q^d_{u, R} \in \mathcal R$.

    Appealing to the definition of $\Lambda_R^d$, we establish for any $v \in R\mathbb Z^d$
    \begin{equation}\label{eq:eq_}
        \begin{split}
            \Lambda_R^d (v, x)&=\Lambda_R^d (v, Rw+(x-Rw))\\
            &=\left(\frac x R -w\right)^{\left(\frac v R -w\right)}\\
            &=\Lambda_R^d (v+(y-x), (Rw+(y-x))+(x-Rw))\\
            &=\Lambda_R^d (v+(y-x), y),
        \end{split}
    \end{equation}
    where $v+(y-x)\in R\mathbb Z^d$ by the assumption; this verifies the first claim.

    As $\Lambda^d_R(v, y) = 0$ whenever $v\in\mathcal V^d_R$, $v\notin \mathcal V(Q^d_{u, R}) \subseteq V$, we deduce $\{ v+(y-x) : v\in V,\, \Lambda^d_R(v, x)\neq 0\} \subseteq \{ v : v\in V,\,\Lambda^d_R(v, y)\neq 0 \}$; this proves the second claim. 
    
    In fact, since the inclusion similarly holds for the role of $x$, $y$ interchanged, we conclude \begin{equation} \label{eq:eq} \{ v+(y-x) : v\in V,\, \Lambda^d_R(v, x)\neq 0\} = \{ v : v\in V,\,\Lambda^d_R(v, y)\neq 0 \} \text.\end{equation}

    We assert that
    \begin{align*}
        r(y) &= \sum_{\substack{v\in V\\\Lambda^d_R(v, y)\neq 0}} \Lambda^d_R(v, y)\delta_V(v)\\
        &=\sum_{\substack{v\in V\\\Lambda^d_R(v, x)\neq 0}} \Lambda^d_R(v+(y-x), y)\delta_V(v+(y-x))\\
        &=\sum_{\substack{v\in V\\\Lambda^d_R(v, x)\neq 0}} \Lambda^d_R(v, x)\delta_V(v+(y-x)).
    \end{align*}
    Indeed, the first equality follows from the definition of $r$ while for the second and third one the equations~\eqref{eq:eq} and~\eqref{eq:eq_} apply, respectively. The proof is now complete.
\end{proof}

\smallskip
We proceed to the proof of the main result. To that end, let us first introduce the following notation.

\begin{notation}
    For any $j\in \{1,\ldots, d\}$, we define $\pi_j : \mathbb R^d \to \mathbb R^{d-1}$, $(x_i)_{i=1}^d \mapsto (x_1, \ldots, x_{j-1}, x_{j+1} \ldots, x_d)$, where $(x_i)_{i=1}^d \in \mathbb R^d$.
    
    Given $v=(v_i)_{i=1}^{d-1} \in\{ 0, 1\}^{d-1}$ we shall write $v^0=(v^0_i)_{i=1}^{d}$, $v^1=(v^1_i)_{i=1}^{d}$ for the elements of $\mathbb \{0, 1\}^{d}$ satisfying $v^0_i=v^1_i=v_i$, where $i \in \{1, \ldots, d-1\}$, and $v^0_d=0$, $v^1_d=1$, respectively.
\end{notation}

\begin{proof}[Proof of~\Cref{thm:lattice_embedding}]
    We claim that up to a~dilation and a~translation, it suffices to consider the case $R=1$ and $Q_{0, 1} \in \mathcal R$. Indeed, if $\mathcal R^\prime \subseteq \mathcal Q^d_R$, $K^\prime = \bigcup_{Q\in\mathcal R^\prime} Q$, $V^\prime = \bigcup_{Q\in \mathcal R^\prime} \mathcal V(Q)$, and $0_{V'}$ is the base point of $V'$, we may find $\mathcal R \subseteq \mathcal Q^d_1$, $Q_{0, 1} \in \mathcal R$, $K  = \bigcup_{Q\in\mathcal R} Q$, $V = \bigcup_{Q\in \mathcal R} \mathcal V(Q)$, $0_V\in V$, and $w\in\mathbb Z^d$, such that $K^\prime=RK+Rw$, $V^\prime=RV+Rw$, and $0_{V'}=R0_V+Rw$. We denote $\sigma : K \to K'$, $y \mapsto Ry+Rw$, where $y\in K$, and consider $0_V$ as the base point of $V$.
    
    We let $\tau: \mathcal F_p (V) \to \mathcal F_p (V^\prime)$ denote the isomorphism from~\Cref{lemma:isometry_of_translation}. Pick $y\in K$. For any $v\in\mathbb Z^d$, we note that $\Lambda_R^d (Rv+Rw, Ry+Rw)=\Lambda_1^d (v+w, y+w)$ and $\Lambda_1^d (v+w, y+w)=\Lambda_1^d (v, y)$ by the defining property of $\Lambda^d_R$ and~\Cref{lemma:translation}~\labelcref{it:translation_1}, respectively. We may thus write
    \begin{equation*}
        \begin{split}
            r_{K^\prime, V^\prime} (\sigma(y)) &= \sum_{v\in V^\prime} \Lambda^d_R(v, \sigma (y))\delta_{V^\prime}(v)\\
            &= \sum_{v\in V} \Lambda^d_R(Rv+Rw, Ry+Rw)\delta_{V^\prime}(Rv+Rw)\\
            &= \sum_{v\in V} \Lambda^d_1(v, y)\delta_{V^\prime}(Rv+Rw)\\
            &= \tau(r_{K, V} (y)) \text.
        \end{split}
    \end{equation*}
    
    Since $y\in K$ was arbitrary, and using the fact that $\left\Vert\tau(x)\right\Vert_{\mathcal F_p (V^\prime)} = R\left\Vert x \right\Vert_{\mathcal F_p (V)}$, where $x\in \mathcal F_p (V)$, we deduce for any $y,\,z\in K$
    \[
        \left\Vert r_{K^\prime, V^\prime} (\sigma(y)) - r_{K^\prime, V^\prime} (\sigma(z)) \right\Vert_{\mathcal F_p (V^\prime)}=R\left\Vert r_{K, V} (y) - r_{K, V} (z) \right\Vert_{\mathcal F_p (V)}\text.
    \]
    
    Similarly, we note that $|\sigma(y)-\sigma(z)|_1 = R|y-z|_1$ for any $y,\,z\in K$. As $\sigma$ is a~bijection of $K$ and $K^\prime$, the claim follows.
    \smallskip
    
    To establish an~upper bound on the Lipschitz constant of $r$, we begin with the case when $x,\, y\in Q$ for some $Q\in \mathcal R$. To that end, let us first note that for $x=(x_i)_{i=1}^d \in Q^d_{w, 1}\in \mathcal R$, where $w=(w_i)_{i=1}^d\in \mathbb Z^d$, it holds by~\Cref{lemma:cf:coeff_Lambda_properties}~\labelcref{it:Lambda_1,it:Lambda_4}
    \begin{equation}\label{eq:coordinatewise_expanded}
        \begin{split}
            r(x) &= \sum_{v\in\mathcal V(Q^d_{w, 1})} \Lambda^d_1(v, x)\delta_V(v) \\
            &=\sum_{u\in \{0, 1\}^{d-1}} \Lambda^{d-1}_1(\pi_d(w)+u, \pi_d(x))\Lambda^1_1(w_d, x_d)\delta_V(w+u^0)\\
            &\phantom{=}+\sum_{u\in \{0, 1\}^{d-1}} \Lambda^{d-1}_1(\pi_d(w)+u, \pi_d(x))\Lambda^1_1(w_d+1, x_d)\delta_V(w+u^1)\\
            &=\sum_{u\in \{0, 1\}^{d-1}} \Lambda^{d-1}_1(\pi_d(w)+u, \pi_d(x))(1-(x_d-w_d))\delta_V(w+u^0)\\
            &\phantom{=}+\sum_{u\in \{0, 1\}^{d-1}} \Lambda^{d-1}_1(\pi_d(w)+u, \pi_d(x))(x_d-w_d)\delta_V(w+u^1).
        \end{split}
    \end{equation}

    Pick $x=(x_i)_{i=1}^d,\, y=(y_i)_{i=1}^d \in Q^d_{w, 1}\in \mathcal R$, where $w=(w_i)_{i=1}^d\in \mathbb Z^d$. We shall further assume that $x$, $y$ differ in at most one coordinate; without loss of generality, let $x_i=y_i$ for any $i\in\{1, \ldots, d-1\}$. It follows from~\eqref{eq:coordinatewise_expanded} that
    \begin{multline}\label{eq:diff_expanded}
        r(x) - r(y) = (x_d-y_d)\cdot\\\sum_{u\in \{0, 1\}^{d-1}}\Lambda^{d-1}_1(\pi_d(w)+u, \pi_d(x))(\delta_V(w+u^1)-\delta_V(w+u^0)) \text.
    \end{multline}

    We have $\sum_{u\in\mathcal V^{d-1} (\pi_d(Q^d_{w, 1}))}\Lambda^{d-1}_1(u, \pi_d(x))=1$ by~\Cref{lemma:cf:coeff_Lambda_properties}~\labelcref{it:Lambda_3}. Since also $\left\Vert\delta_V(w+u^1)-\delta_V(w+u^0)\right\Vert_p=1$ for any $u\in\{0,1\}^{d-1}$, we deduce the inequality \begin{equation} \label{eq:ineq1} \left\Vert r(x)-r(y) \right\Vert_p \leq C(p, 2^{d-1})|x_d-y_d| \text.\end{equation}

    For any $x=(x_i)_{i=1}^d,\, y=(y_i)_{i=1}^d \in Q^d_{w, 1}\in \mathcal R$ we may now take a~sequence $(z^i)_{i=0}^d = ((z^i_j)_{j=1}^d)_{i=0}^d \in (Q^d_{w, 1})^{d+1}$ such that for any $i\in \{0, \ldots, d\}$ and $j\in \{1, \ldots, d\}$, it holds that $z^i_j=x_j$ if and only if $i\geq j$ and $z^i_j=y_j$ otherwise. Any two consecutive elements of the sequence $(z^i)_{i=1}^d$ differ in at most one coordinate, and using the inequality~\eqref{eq:ineq1} we establish
    \begin{equation*}
        \begin{split}
            \left\Vert r(x)-r(y) \right\Vert^p_p &\leq \sum_{i=0}^{d-1} \left\Vert r(z^{i+1})-r(z^i) \right\Vert^p_p \\
            &\leq C^p(p, 2^{d-1}) \sum_{i=1}^{d} |x_i-y_i|^p,
        \end{split}
    \end{equation*}
    hence $\left\Vert r(x)-r(y) \right\Vert_p \leq C(p, 2^{d-1})C(p, d)|x-y|_1$.

    As for the general case, we consider $x=(x_i)_{i=1}^d \in Q^d_{w, 1}\in\mathcal R$ and $y=(y_i)_{i=1}^d \in Q^d_{u, 1}\in\mathcal R$, where $w,\, u \in \mathbb Z^d$. Let us denote $I = \{ i \in \{1, \ldots, d\} : w_i \neq u_i\}$.

    Given $i \in I$, we may find $n_i,\, m_i \in \mathbb Z$ such that $n_i\in \{w_i, w_i+1\}$, $m_i\in \{u_i, u_i+1\}$, and $|x_i-y_i|=|x_i-n_i|+|n_i-m_i|+|m_i-y_i|$. Let us also pick $x^\prime = (x^\prime_i)_{i=1}^d \in \mathbb R^d$, $y^\prime = (y^\prime_i)_{i=1}^d \in \mathbb R^d$ such that $x^\prime_i=y^\prime_i=x_i$ if and only if $i\in\{1, \ldots, d\} \setminus I$ and $x^\prime_i=n_i$, $y^\prime_i=m_i$ otherwise, respectively; it follows after short thought that $x^\prime \in Q^d_{w, 1}$, $y^\prime \in Q^d_{u, 1}$, and $|x-y|_1=|x-x^\prime|_1+|x^\prime-y^\prime|_1+|y^\prime-y|_1$.

    We note that $y^\prime-x^\prime \in \mathbb Z^d$; hence~\Cref{lemma:translation} applies and
    \begin{gather*}
        r(y^\prime) = \sum_{\substack{v\in V\\\Lambda^d_1(v, x^\prime)\neq 0}} \Lambda^d_1(v, x^\prime)\delta_V(v+(y^\prime-x^\prime)),\\
        r(x^\prime)-r(y^\prime) = \sum_{\substack{v\in V\\\Lambda^d_1(v, x^\prime)\neq 0}} \Lambda^d_1(v, x^\prime)(\delta_V(v)-\delta_V(v+(y^\prime-x^\prime))) \text.
    \end{gather*}
    
    If $I$ was empty, then necessarily $r(x^\prime)=r(y^\prime)$. Otherwise at least one coordinate of $x^\prime$ assumes value in $\mathbb Z$; we claim $|\{ v\in\mathcal V(Q): \Lambda^d_1(v, x^\prime) \neq 0 \}| \leq 2^{d-1}$ in this case. To that end, let $x'_j \in \mathbb Z$ for some $j\in \{1, \ldots, d\}$, and define $w'_j=w_j+1$ if and only if $x'_j=w_j+1$ and $w'_j=w_j-1$ otherwise. If we now denote $w' = (w_1, \ldots, w_{j-1}, w'_j, w_{j+1}, \ldots, w_d)\in \mathbb Z^d$, it follows that $w\neq w'$ and $x^\prime \in Q_{w, 1} \cap Q_{w', 1}$. We have that $\Lambda^d_1 (v, x') = 0$ whenever $v\notin \mathcal V(Q_{w, 1}) \cap \mathcal V(Q_{w', 1})$ by~\Cref{lemma:cf:coeff_Lambda_properties}~\labelcref{it:Lambda_1}; since $|\mathcal V(Q_{w, 1}) \cap \mathcal V(Q_{w', 1})|\leq 2^{d-1}$, the claim follows.
    
    As $\left\Vert\delta_V(v)-\delta_V(v+(y^\prime-x^\prime))\right\Vert_p=|x^\prime-y^\prime|_1$ for any $v\in\mathcal V(Q)$, $\Lambda^d_1(v, x^\prime)\neq 0$, and $\sum_{v\in V} \Lambda^d_1(v, x^\prime) = 1$ by~\Cref{lemma:cf:coeff_Lambda_properties}~\labelcref{it:Lambda_3}, in either case we are thus justified to establish the inequality \[ \left\Vert r(x^\prime)-r(y^\prime) \right\Vert_p \leq C(p, 2^{d-1}) |x^\prime-y^\prime|_1 \text. \]

    Using the already proven parts and the fact that $x,\,x^\prime\in Q^d_{w, 1}$, $y,\,y^\prime\in Q^d_{u, 1}$, we deduce
    \begin{align*}
        \left\Vert r(x)-r(y) \right\Vert^p_p &\leq \left\Vert r(x)-r(x^\prime) \right\Vert^p_p + \left\Vert r(x^\prime)-r(y^\prime) \right\Vert^p_p + \left\Vert r(y^\prime)-r(y) \right\Vert^p_p \\
        &\leq C^p(p, 2^{d-1})C^p(p, d) (|x-x^\prime|^p_1+|x^\prime-y^\prime|^p_1+|y^\prime-y|^p_1),
    \end{align*} 
    and altogether we obtain $\left\Vert r(x)-r(y) \right\Vert_p \leq C(p, 2^{d-1})C(p, d)C(p, 3)|x-y|_1$.
    \smallskip

    Regarding the converse part, we let $\mathcal R \subseteq \mathcal Q^d_R$, where $\mathcal R \neq \emptyset$, $K = \bigcup_{Q\in\mathcal R} Q$, $V = \bigcup_{Q\in \mathcal R} \mathcal V(Q)$, and consider $0_V$ as the base point of $V$. By the initial remark, it suffices to consider the case $R=1$ and $Q_{0, 1} \in \mathcal R$.
    
    We let $x = (x_i)_{i=1}^d$, $y = (y_i)_{i=1}^d$ be such that $x_i=y_i=1/2$, for each $i\in\{1, \ldots, d-1\}$, and $x_d=0$, $y_d=1$. Recall that using~\eqref{eq:diff_expanded}, we have $r(y)-r(x) = \sum_{u\in\mathcal \{0, 1\}^{d-1}} 2^{-d+1}(\delta_V(u^1)-\delta_V(u^0))$.

    Denote $\mathcal M = \left\{ \frac{\delta(z)-\delta(z')}{|z-z'|_1}: z,\,z' \in V,\, z\neq z' \right\}$; by~\Cref{fact:cf:isometrically_norming}, we have \[ \left\Vert r(y)-r(x) \right\Vert_p = \inf \left( \sum_{i=1}^n |a_i|^p \right)^{1/p} \text,\] the infimum being taken over all $n\in\mathbb N_0$ and $\mu_i\in\mathcal M$, $a_i\in\mathbb R$, where $i\in\{1, \ldots, n\}$, such that $r(y)-r(x) = \sum_{i=1}^n a_i\mu_i$.

    We shall see that $\left\Vert r(x) - r(y) \right\Vert_p \geq C(p, 2^{d-1})|x-y|_1 = C(p, 2^{d-1})$. To that end, pick $n\in\mathbb N_0$ and $a_i\in\mathbb R$, $\mu_i\in\mathcal M$, where $i \in \{1, \ldots, n\}$, as above.
    
    We introduce $\mathcal N : \mathcal V(Q_{0, 1}) \to \mathcal P (\{1, \ldots, n\})$ as \[v \mapsto \left\{ i \in \{1, \ldots, n\} : \mu_i \in \left\{\pm\frac{\delta(z)-\delta(v)}{|z-v|_1}:\,x \in V,\,z\neq v \right\} \right\}\text.\]

    It follows that for any $i \in \{1, \ldots, n\}$, there exist at most two distinct elements $u,\, v \in \mathcal V(Q_{0, 1})$ such that $i \in \mathcal N(u) \cap \mathcal N(v)$; consequently, $\sum_{i=1}^n |a_i|^p \geq \frac 1 2 \sum_{u\in\mathcal V(Q_{0, 1})} \sum_{i\in\mathcal N(u)} |a_i|^p$.
    
    We claim that $\sum_{i\in\mathcal N(u)} |a_i| \geq 2^{-d+1}$ for any $u \in \mathcal V(Q_{0, 1})$.
    
    To that end, we pick $u \in \mathcal V(Q_{0, 1})$ and construct $\varphi : V\to\mathbb R$ as follows: If $\chi_{A}$ denotes the indicator function of a~set $A\subseteq V$, we define $\varphi = \chi_{\{u\}}$ if $u\neq 0_V$ and $\varphi=\chi_{V\setminus \{ u \}}$ otherwise. We note that $\varphi \in \Lip_0 (V)$, $\Lip \varphi \leq 1$, as $\varphi(0_V)=0$ and $|\varphi(z)-\varphi(z)|\leq 1$, $|z-z'|_1 \geq 1$ for any two $z,\, z' \in V$, $z\neq z'$. By~\cref{it:F_p_extension} of $\mathcal F_p (V)$, we let $\varphi^*\in \mathcal F_p (V)^*$ be such that $\varphi=\varphi^*\circ\delta_V$ on the set $V$. 

    We observe that $\left|\left\langle \varphi^*, \mu_i \right\rangle\right| \leq 1$ for any $i \in \{1, \ldots, n\}$, where $\left\langle \varphi^*, \mu_i \right\rangle = 0$ for each $i \in \{1, \ldots, n\} \setminus \mathcal N (u)$. Hence, we may write
    \begin{gather*}
        \begin{split}
            \left|\left\langle \varphi^*, r(y)-r(x) \right\rangle\right|&=\left|\left\langle \varphi^*,\sum_{i=1}^n a_i\mu_i \right\rangle\right|\\
            &\leq \sum_{i\in\mathcal N(u)} |a_i|\left|\left\langle \varphi^*, \mu_i \right\rangle\right|\leq \sum_{i\in\mathcal N(u)} |a_i| \text,
        \end{split}\\
        \intertext{and}
        \begin{split}
        \left|\left\langle \varphi^*,r(y)-r(x) \right\rangle\right|&=\left|\left\langle \varphi^*, \sum_{v\in\mathcal \{0, 1\}^{d-1}} 2^{-d+1}(\delta_V(v^1)-\delta_V(v^0)) \right\rangle\right|\\
        &=2^{-d+1} \text,
        \end{split}
    \end{gather*}
    and the intermediate claim is established.
    
    We note that $x\mapsto x^p$ is subadditive on $[0, \infty)$, and thus $\sum_{i\in\mathcal N(u)} |a_i|^p \geq 2^{p(-d+1)}$ for any $u \in \mathcal V(Q_{0, 1})$. Altogether, we are justified to establish
    \begin{equation*}
        \begin{split}
            \sum_{i=1}^n |a_i|^p &\geq \frac 1 2 \sum_{u\in\mathcal V(Q_{0, 1})} \sum_{i\in\mathcal N(u)} |a_i|^p\\
            &\geq 2^{d-1}2^{p(-d+1)}\\
            &=C^p(p, 2^{d-1}),
        \end{split}
    \end{equation*}
    which concludes the proof.
\end{proof}
\smallskip

Let us recall a~closely related result which shows that $\mathcal F_p ([0, 1]^d)$ has a~Schauder basis, consider~\cite[Theorem~3.8]{Albiac2022}. In fact, the associated canonical projections are the retractions $r$; hence, \Cref{thm:lattice_embedding} gives a~refined estimate on the basis constant.

\section{On the~Geometry of~\texorpdfstring{$\mathcal F_p (\mathcal M, \rho^\alpha)$}{Fp(M, rho alpha)}, Where~\texorpdfstring{$(\mathcal M, \rho)$}{(M, rho)} Is Infinite Doubling}\label{sec:isomorphism}

We recall that a~metric space $(\mathcal M, \rho)$ is \emph{doubling} if there exists $\lambda_{\mathcal M} \in \mathbb N$, called the \emph{doubling constant}, such that any closed ball in $\mathcal M$ of radius $2r>0$ can be covered by $\lambda_{\mathcal M}$-many closed balls of radius $r$. As an~example, any subspace of a~finite-dimensional Banach space is doubling.

A~classical result due to~\textcite{Assouad1983} implies that for any $\alpha \in (0, 1)$, the distorted space $(\mathcal M, \rho^\alpha)$ is bi-Lipschitz equivalent to $(\mathcal M', |\cdot|^\beta)$ for some $\mathcal M' \subseteq \mathbb R^d$, where $d\in\mathbb N$ and $\beta \in (0,1)$. As a~consequence, it follows that $\mathcal F_p (\mathcal M, \rho^\alpha) \simeq \mathcal F_p (\mathcal M', | \cdot|^\beta)$ for any $0<p\leq 1$, see~\Cref{fact:F_p_linearization}.

Moreover, it was observed in~\cite{Albiac2021sums} that for infinite subsets $\mathcal M$ of $\mathbb R^d$, the isomorphism theorem reduces to the case when $\mathcal M= [0, 1]^d$. Hence, the central topic of this section is to fill in the missing part and establish that $\mathcal F_p ([0, 1]^d, |\cdot|^\alpha)\simeq \ell_p$, for any $0 < \alpha < 1$ and $0<p\leq 1$.

\subsection{The Isomorphism~\texorpdfstring{$\mathcal F_p ([0, 1]^d, |\cdot|^\alpha)\simeq \ell_p$}{Fp ([0, 1] d, |.|alpha)}}\label{subsec:isomorphism_cube}

It is easy to see that for any fixed $0<\alpha< 1$, the associated \emph{Hölder distortions} of any two norms $|\cdot|$, $|\cdot|_*$ on $\mathbb R^d$ are Lipschitz equivalent; thus, $\mathcal F_p ([0, 1]^d, |\cdot|^\alpha) \simeq \mathcal F_p ([0, 1]^d, |\cdot|_*^\alpha)$ for any $0<p\leq 1$, again by~\Cref{fact:F_p_linearization}. It will be convenient to identify $|\cdot|$ as the $\ell_1$ norm in the sequel.

\begin{thm}\label{thm:isomorphism_p_cube}
    Let $d\in\mathbb N$ and $0 < \alpha < 1$, $0<p\leq 1$. Then $\mathcal F_p ([0, 1]^d, |\cdot|^\alpha)$ is isomorphic to the space~$\ell_p$.
\end{thm}

Quantitatively, if we consider the Banach-Mazur distance (see, e.g.,~\cite[p.~277]{Jaegermann1989}) of $\mathcal F_p ([0, 1]^d, |\cdot|^\alpha)$ and $\ell_p (V)$ defined by $d(\mathcal F_p ([0, 1]^d, |\cdot|^\alpha), \ell_p (V)) = \inf\, \lVert T \rVert \lVert T^{-1}\rVert$, where $T$ ranges over onto isomorphisms $T: \mathcal F_p ([0, 1]^d, |\cdot|^\alpha) \to \ell_p (V)$, we show that $d(\mathcal F_p ([0, 1]^d, |\cdot|^\alpha), \ell_p (V)) \leq C(p, 2^d)\rho^d\tau^d$, where $\rho$ and $\tau$ are as in~\Cref{lemma:combinations,lemma:step}, respectively.

\smallskip
Pick $0 < \alpha < 1$ and $0<p\leq 1$. Whenever $d\in\mathbb N$ is fixed and known from the context, we shall denote $V_{-1} = \{ 0 \}$, $V_k = [0, 1]^d \cap 2^{-k}\mathbb Z^d$, where $k\in\mathbb N_0$, and $V = \bigcup_{k\in\mathbb N_0} V_k\setminus V_{k-1}$. We take $0$ as the base point of $[0, 1]^d$ and write $\delta : [0, 1]^d \to \mathcal F_p ([0, 1]^d, |\cdot|^\alpha)$ for the canonical isometric embedding.

\smallskip
The proof proceeds by considering the~map $\iota : \{ e_v : v\in V \} \subset \ell_p (V) \to \mathcal F_p ([0, 1]^d, |\cdot|^\alpha)$, \[e_v \mapsto 2^{k\alpha}\left(\delta(v)-\sum_{u \in V_{k-1}} \Lambda^d_{2^{-k+1}}(u, v)\delta(u)\right), \quad v\in V_k\setminus V_{k-1}, \, k\in\mathbb N_0 \text. \]

A~short thought shows that $\{ e_v : v\in V\}$ is $p$-norming (see~\Cref{def:p_norming}) in $\ell_p (V)$ and that $\iota$ extends to a~one-to-one linear map from $\operatorname{span} \{ e_v : v\in V\}$ into $\mathcal F_p ([0, 1]^d, |\cdot|^\alpha)$. Hence, by~\Cref{fact:extension}, $\iota$ shall provide us with an~onto isomorphism $\tilde\iota : \ell_p(V)\to \mathcal F_p ([0, 1]^d, |\cdot|^\alpha)$ once we show that $\iota(\{e_v : v\in V\})$ is $p$-norming in $\mathcal F_p ([0, 1]^d, |\cdot|^\alpha)$.

Let us note this particular choice of $\iota(e_v)$, where $v\in V$, is frequent in the theory of Lipschitz free spaces over an~Euclidean space. Indeed, a~dual variant thereof is used in the standard proof of the $p=1$ case (see, e.g.,~\cite[Theorem~8.44]{Weaver2018}), and it is known that $\iota(e_v)$, where $v\in V$, forms a~Schauder basis of~$\mathcal F_p ([0, 1]^d, |\cdot|)$, see~\cite[Theorem~3.8]{Albiac2022}.

\begin{notation}
    If $d\in\mathbb N$ is fixed, $x=(x_i)_{i=1}^d \in \mathbb R^d$, $j\in\{1, \ldots, d\}$, and $\epsilon\in\mathbb R$, let us set $x^j_\epsilon=(x_1, \ldots, x_j+\epsilon, \ldots, x_d)\in \mathbb R^d$. If $x\in [0, 1]$, $x\in 2^{-n}\mathbb Z \setminus 2^{-n+1}\mathbb Z$, for some $n\in\mathbb N$, we will write $x^+ = x+2^{-n}$, $x^- = x-2^{-n}$.
    
    For $d\in\mathbb N$, $j\in\{1, \ldots, d\}$, and $(x_i)_{i=1}^{d-1}\in [0, 1]^{d-1}$, we furthermore adopt the notation $\delta^j_{(x_i)_{i=1}^{d-1}} (x) = \delta (x_1, \ldots, x_{j-1}, x, x_j, \ldots, x_{d-1})$.
\end{notation}

\subsubsection{One~Linearization Result}

We expand an~element $\delta(v)$, where $v\in V$, in a~way which partially recovers the candidate basis geometry. This will have a~great importance for the~construction considered in~\Cref{lemma:step}.

\begin{lemma} \label{lemma:hat}
    Let $u_1, u_2 \in [0, 1]\cap 2^{-n}\mathbb Z$ for some $n\in\mathbb N_0$, $u_1<u_2$, and $|u_1-u_2|=2^{-n}$. Let $v\in [u_1, u_2]$ satisfy $v\in 2^{-k}\mathbb Z$, where $k\in\mathbb N$, $k\geq n$.

    There exist $\mu_1,\, \mu_2 \geq 0$, $\mu_1+\mu_2=1$, $l \in \mathbb N_0$, $\nu_i > 0$, $n_i \in \mathbb N$, $n_i>n$, and $v_i \in 2^{-n_i}\mathbb Z \setminus 2^{-n_i+1}\mathbb Z$, where $i \in \{1, \ldots, l\}$, such that
    \[ 
        \begin{split} 
            2^{n\alpha}\delta^j_x(v) &= \mu_1 2^{n\alpha}\delta^j_x(u_1)+\mu_2 2^{n\alpha}\delta^j_x(u_2)\\ 
            &\phantom{=}+\sum_{i=1}^l \nu_i 2^{n_i\alpha}\left(\delta^j_x (v_i) - \frac 1 2 (\delta^j_x (v_i^-)+\delta^j_x (v_i^+))\right)\text,
        \end{split}
    \]
    for any $d\in\mathbb N$, $j\in\{1, \ldots, d\}$, and $x=(x_j)_{j=1}^{d-1}\in [0, 1]^{d-1}$.

    Moreover, we have $\left(\sum_{i=1}^l |\nu_i|^p\right)^{1/p} \leq 2^{-\alpha}\left(\frac 1 {1-2^{-p\alpha}}\right)^{1/p}$.
\end{lemma}

\begin{proof}
    We shall establish that for any $k\in\mathbb N$, $k\geq n$, and for any $v\in [u_1, u_2] \cap 2^{-k}\mathbb Z$, we can choose the coefficients $\mu^v_1,\, \mu^v_2 \geq 0$, $l^v \in \mathbb N_0$, $\nu^v_i \in [0, 1]$, $v^v_i \in [u_1, u_2]$, and $n^v_i \in \mathbb N$, where $i\in \{1, \ldots, l^v\}$, so that for any $v\in [u_1, u_2] \cap 2^{-k}\mathbb Z$, it holds
    \begin{enumerate}[label=(\roman*), ref=\roman*]
        \item\label{it:hat_cd_mu} $\mu^v_1, \, \mu^v_2 \geq 0$, $\mu^v_1+\mu^v_2=1$, and $k\geq n^v_i > n$ for every $i \in \{1, \ldots, l^v\}$,
        \item\label{it:hat_cd_once} if $v^\prime \in [u_1, u_2] \cap 2^{-k}\mathbb Z$, there exists at most one index $i\in\{1, \ldots, l^v\}$ such that $v^v_i=v^\prime$,
        \item\label{it:hat_cd_where} $v^v_i \in [u_1, u_2] \cap 2^{-n^v_i}\mathbb Z \setminus 2^{-n^v_i+1}\mathbb Z$ for every $i \in \{1, \ldots, l^v\}$,
        \item\label{it:hat_cd_delta} for any $j\in\{1, \ldots, d\}$ and $x=(x_j)_{j=1}^{d-1}\in [0, 1]^{d-1}$, we have 
        \[ 
            \begin{split} 
                2^{n\alpha}\delta^j_x(v) &= \mu^v_1 2^{n\alpha}\delta^j_x(u_1)+\mu^v_2 2^{n\alpha}\delta^j_x(u_2)\\ 
                &\phantom{=}+\sum_{i=1}^l \nu^v_i 2^{n^v_i\alpha}\left(\delta^j_x (v^v_i) - \frac 1 2 (\delta^j_x ((v^v_i)^-)+\delta^j_x ((v^v_i)^+))\right),
            \end{split}
        \]
        \item\label{it:hat_cd_ineq} $0<\nu^v_i\leq 2^{(n-n^v_i)\alpha}$ for every $i \in \{1, \ldots, l^v\}$,
        \item\label{it:hat_cd_inclusion} if $v^\prime \in [u_1, u_2] \cap 2^{-k+1}\mathbb Z$ is such that $|v-v^\prime|=2^{-k}$, then $\{ v^{v^\prime}_i : i\in \{1, \ldots, l^{v^\prime}\}\} \subseteq \{ v^{v}_i : i\in \{1, \ldots, l^{v}\}\}$,
        \item\label{it:hat_cd_count} $|\{ v^{v}_i : i\in \{1, \ldots, l^{v}\}\} \cap 2^{-m}\mathbb Z \setminus 2^{-m+1}\mathbb Z\} | \leq 1$ for any $m \in \mathbb N$, $m>n$.
    \end{enumerate}
   
    We proceed by induction on $k$. To that end, note that for $k=n$, we may take $l^v=0$ and $\mu^v_1=1$, $\mu^v_2=0$ or $\mu^v_1=0$, $\mu^v_2=1$ for $v$ equal to $u_1$ or $u_2$, respectively.

    Let $k\in \mathbb N$, $k>n$, be such that the claim holds for $k-1$. 
    
    For any $v\in [u_1, u_2] \cap 2^{-k}\mathbb Z$, we define the coefficients $\mu^v_1,\,\mu^v_2 \in \mathbb R$, $l^v \in \mathbb N_0$, $\nu^v_i \in [0, 1]$, $n^v_i \in \mathbb N$, $k\geq n^v_i>n$, $i \in \{1, \ldots, l^v\}$, as follows. If $v\in 2^{-k+1}\mathbb Z$, we take the coefficients from the induction hypothesis. Otherwise $v\in 2^{-k} \mathbb Z \setminus 2^{-k+1}\mathbb Z$, and thus $v^-,\, v^+ \in [u_1, u_2] \cap 2^{-k+1}\mathbb Z$. For any $d\in\mathbb N$, $j\in\{1, \ldots, d\}$, and $x=(x_j)_{j=1}^{d-1}\in [0, 1]^{d-1}$, we may write
    \begin{equation}\label{eq:hat_inductive_expansion}
        \begin{split}
            2^{n\alpha}\delta^j_x(v)&= 2^{(n-k)\alpha}2^{k\alpha}\left(\delta^j_x (v) - \frac 1 2 (\delta^j_x (v^-)+\delta^j_x (v^+))\right)\\
            &\phantom{=}+ 2^{n\alpha-1}\delta^j_x(v^-)+ 2^{n\alpha-1}\delta^j_x(v^+) \text.
        \end{split}
    \end{equation}

    Let now $\mu^{v^+}_1$, $\mu^{v^+}_2$, $l^{v^+}$, $\nu^{v^+}_i$, $n^{v^+}_i$, $v^{v^+}_i$, $i \in \{1, \ldots, l^{v^+}\}$, and $\mu^{v^-}_1, \mu^{v^-}_2$, $l^{v^-}$, $\nu^{v^-}_i$, $n^{v^-}_i$, $v^{v^+}_i$, $i \in \{1, \ldots, l^{v^-}\}$, be the coefficients corresponding to $v^+$, $v^-$. We denote $\mathcal V = \{ v^{v^+}_i : i\in \{1, \ldots, l^{v^+}\}\} \cup \{ v^{v^-}_i : i\in \{1, \ldots, l^{v^-}\}\}$ and set $l^v = |\mathcal V|+1$. Let $v^v_i$, $i\in \{1, \ldots, l^v-1\}$, be such that $\{ v^v_i : i\in\{1, \ldots, l^v-1\}\}=\mathcal V$; such an arrangement of the finite set $\mathcal V$ necessarily exists. By the choice of $\mathcal V$ and the induction hypothesis, for any $i\in \{1, \ldots, l^v-1\}$ we may further find $n^v_i \in\mathbb N$, $k-1\geq n^v_i > n$, such that $v^v_i\in 2^{-n^v_i}\mathbb Z \setminus 2^{-n^v_i+1}\mathbb Z$.  If $i\in \{1, \ldots, l^v-1\}$, let us further denote $\nu^v_i = \frac 1 2\sum_{j\in \{1, \ldots, l^{v^+}\},\,v^{v^+}_j=v^v_i} \nu^{v^+}_j +\frac 1 2\sum_{j\in \{1, \ldots, l^{v^-}\}, \,v^{v^-}_j=v^v_i} \nu^{v^-}_j$; it follows from the induction hypothesis that there is at most one index $j\in \{1, \ldots, l^{v^+}\}$ such that $v^{v^+}_j=v^v_i$ (similarly for $v^-$), and thus we obtain $0<\nu^v_i \leq 2^{\alpha (n-n^v_i)}$. 

    We set $n^v_{l^v}=k$, $v^v_{l^v}=v$, $\nu^v_{l^v}=2^{(n-k)\alpha}$, and $\mu^v_1=\frac 1 2 (\mu^{v^+}_1+\mu^{v^-}_1)$, $\mu^v_2=\frac 1 2 (\mu^{v^+}_2+\mu^{v^-}_2)$. Let us remark that $\mu^v_1,\, \mu^v_2 \geq 0$ and $\mu^v_1+\mu^v_2=\frac 1 2 (\mu^{v^+}_1+\mu^{v^-}_1)+\frac 1 2 (\mu^{v^+}_2+\mu^{v^-}_2)= 1$ since $\mu^{v^+}_1+\mu^{v^+}_2 = \mu^{v^-}_1+\mu^{v^-}_2=1$, by the induction hypothesis.
    
    Continuing~\eqref{eq:hat_inductive_expansion}, the above coefficients let us rewrite
    \[
        \begin{split}
            2^{n\alpha}\delta^j_x(v)&= 2^{(n-k)\alpha}2^{k\alpha}\left(\delta^j_x (v) - \frac 1 2 (\delta^j_x (v^-)+\delta^j_x (v^+))\right)\\
            &\phantom{=}+ 2^{n\alpha-1}\delta^j_x(v^-)+ 2^{n\alpha-1}\delta^j_x(v^+)\\
            &= \mu^v_1 2^{n\alpha}\delta^j_x(u_1)+\mu^v_2 2^{n\alpha}\delta^j_x(u_2)\\ 
            &\phantom{=}+\sum_{i=1}^{l^v-1} \nu^v_i 2^{n^v_i\alpha}\left(\delta^j_x (v^v_i) - \frac 1 2 (\delta^j_x ((v^v_i)^-)+\delta^j_x ((v^v_i)^+))\right)\\
            &\phantom{=}+2^{(n-k)\alpha}2^{k\alpha}\left(\delta^j_x (v) - \frac 1 2 (\delta^j_x (v^-)+\delta^j_x (v^+))\right)\\
            &= \mu^v_1 2^{n\alpha}\delta^j_x(u_1)+\mu^v_2 2^{n\alpha}\delta^j_x(u_2)\\
            &\phantom{=}+\sum_{i=1}^{l^v} \nu^v_i 2^{n^v_i\alpha}\left(\delta^j_x (v^v_i) - \frac 1 2 (\delta^j_x ((v^v_i)^-)+\delta^j_x ((v^v_i)^+))\right)\text.
        \end{split}
    \]

    We note that for any $v\in [u_1, u_2] \cap 2^{-k}\mathbb Z$ and for the elements $\mu^v_1$, $\mu^v_2$, $l^v$, $n^v_i$, $v^v_i$, and $\nu^v_i$, where $i \in \{1, \ldots, l^v\}$, the~\crefrange{it:hat_cd_mu}{it:hat_cd_inclusion} now follow at once if $v\in 2^{-k+1}\mathbb Z$ by the induction hypothesis and were otherwise established during the construction whenever $v\in 2^{-k}\mathbb Z \setminus 2^{-k+1}\mathbb Z$.

    To establish~\cref{it:hat_cd_count}, we first remark the conclusion is satisfied for any $v\in 2^{-k+1}\mathbb Z$ by the construction and the induction hypothesis. Let us pick $v\in [u_1, u_2] \cap 2^{-k}\mathbb Z \setminus 2^{-k+1}\mathbb Z$.

    Since we have $v^-,\,v^+ \in 2^{-k+1}\mathbb Z$, $|v^+-v^-|=2^{-k+1}$, it follows that $\{v^-, v^+\} \cap 2^{-k+2}\mathbb Z \neq \emptyset$. Hence, there exist $u$, $u^\prime$ such that $u \in \{v^-, v^+\} \cap 2^{-k+1}\mathbb Z \setminus 2^{-k+2}\mathbb Z$, $u^\prime\in\{v^-, v^+\} \cap 2^{-k+2}\mathbb Z$. By the induction hypothesis and~\cref{it:hat_cd_inclusion} for $u$ and $u^\prime$, we get that $\{ v^{u^\prime}_i : i\in \{1, \ldots, l^{u^\prime}\}\} \subseteq \{ v^{u}_i : i\in \{1, \ldots, l^{u}\}\}$. Since $\{ v^{v}_i : i\in \{1, \ldots, l^{v}\}\} \cap 2^{-k+1}\mathbb Z = \{ v^{v^+}_i : i\in \{1, \ldots, l^{v^+}\}\} \cup \{ v^{v^-}_i : i\in \{1, \ldots, l^{v^-}\}\}$ by the construction, we deduce that \[ \{ v^{v}_i : i\in \{1, \ldots, l^{v}\}\} \cap 2^{-k+1}\mathbb Z = \{ v^{u}_i : i\in \{1, \ldots, l^{u}\}\} \text. \]

    The conclusion of~\cref{it:hat_cd_count} now follows from the induction hypothesis for $u$ whenever $m \in \mathbb N$, $k-1\geq m>n$. Since also $\{v^{v}_i : i\in \{1, \ldots, l^{v}\}\} \setminus 2^{-k+1}\mathbb Z = \{v\}$, ~\cref{it:hat_cd_count} is satisfied. The proof of the induction step is now complete.

    Pick $k\in\mathbb N$, $k\geq n$, and $v\in [u_1, u_2] \cap 2^{-k}\mathbb Z$. We note that $n^v_i > n$, where $i \in \{ 1, \ldots, l^v\}$, by~\cref{it:hat_cd_mu}, and $|\{ i \in \{ 1, \ldots, l^v\} : n^v_i=j \}|\leq 1$, where $j\in \mathbb N$, $j>n$, by~\cref{it:hat_cd_once,it:hat_cd_where,it:hat_cd_count}. Likewise, \cref{it:hat_cd_ineq} shows that $0<\nu^v_i\leq 2^{(n-n^v_i)\alpha}$ for every $i \in \{1, \ldots, l^v\}$; hence,
    \begin{equation*}
        \begin{split}
            \left(\sum_{i=1}^{l^v} |\nu^v_i|^p\right)^{1/p} &\leq \left(\sum_{i=1}^{l^v} 2^{p(n-n^v_i)\alpha}\right)^{1/p}\\
            &< \left(\sum_{i=1}^{\infty} 2^{-p\alpha i}\right)^{1/p} = 2^{-\alpha}\left(\frac 1 {1-2^{-p\alpha}}\right)^{1/p} \text.
        \end{split}
    \end{equation*}

    The claim is established.
\end{proof}
\smallskip

\begin{lemma} \label{lemma:step}
    Let $d\in\mathbb N$ and $\mathcal M \subseteq \mathcal F_p ([0, 1]^d, |\cdot|^\alpha)$ be such that for any $v\in V_k$, where $k\in\mathbb N_0$, we have $2^{k\alpha}\left(\delta(v)-\sum_{u \in V_{k-1}} \Lambda^d_{2^{-k+1}}(u, v)\delta(u)\right) \in \mathcal M$.

    There exists a constant $\rho^\prime>0$ for which the following holds true. Consider $v=(v_i)_{i=1}^d\in V$ and $i\in \{1, \ldots, d\}$. If $v_i \in 2^{-n}\mathbb Z \setminus 2^{-n+1}\mathbb Z$ for some $n\in\mathbb N$, then $2^{n\alpha}\left( \delta(v)-\frac 1 2 (\delta(v^i_{2^{-n}})+\delta(v^i_{-2^{-n}}))\right) \in \rho^\prime\aconv_p \mathcal M$. Quantitatively, if we set $\rho = \left( C^p(p, 2) + (1+2^{1-p})2^{-p\alpha}\left(\frac 1 {1-2^{-p\alpha}}\right)\right)^{1/p}$, then $\rho'=\rho^d$.
\end{lemma}

\begin{proof}
    We prove that for any $l\in \{0, \ldots, d-1\}$, the following claim is true. Let $v=(v_i)_{i=1}^d\in V$. If $i\in \{1, \ldots, d\}$ is such that $v_i \in 2^{-n}\mathbb Z \setminus 2^{-n+1}\mathbb Z$ for some $n\in\mathbb N$ and if $|\{ j\in \{1, \ldots, d\} : v_j \notin 2^{-n}\mathbb Z\}|\leq l$, then $2^{n\alpha}\left( \delta(v)-\frac 1 2 (\delta(v^i_{2^{-n}})+\delta(v^i_{-2^{-n}}))\right) \in \rho^{l+1} \aconv_p \mathcal M$. We proceed by induction on $l$.

    Let $l=0$. We pick $v=(v_i)_{i=1}^d\in V$, $n\in\mathbb N$, and $i\in \{1, \ldots, d\}$ as above. Recall that for any $u=(u_i)_{i=1}^d \in V_{n-1}$, it follows from~\Cref{lemma:cf:coeff_Lambda_properties}~\labelcref{it:Lambda_4} that $\Lambda^d_{2^{-n+1}}(u, v)=\Lambda^{d-1}_{2^{-n+1}}(\pi_d(u), \pi_d(v))\Lambda^1_{2^{-n+1}}(u_i, v_i)$. Hence, $\Lambda^d_{2^{-n+1}}(u, v)=0$ whenever $u_i\notin \{v_i \pm2^{-n}\}$, and now
    \begin{multline*}
            \sum_{u \in V_{n-1}} \Lambda^d_{2^{-n+1}}(u, v)\delta(u)\\
            \begin{aligned}
                &=\sum_{\substack{u=(u_i)_{i=1}^d \in V_{n-1} \\ u_i\in \{v_i \pm2^{-n}\}}} \Lambda^{d-1}_{2^{-n+1}}(\pi_i(u), \pi_i(v))\Lambda^1_{2^{-n+1}}(u_i, v_i)\delta(u)\\
                &=\frac 1 2\sum_{\substack{u=(u_i)_{i=1}^d \in V_{n-1} \\ u_i\in \{v_i \pm2^{-n}\}}} \Lambda^{d-1}_{2^{-n+1}}(\pi_i(u), \pi_i(v))\delta(u) \text.
            \end{aligned}
    \end{multline*}

    Repeating the same argument, we verify 
    \begin{gather*}
        \sum_{u \in V_{n-1}} \Lambda^d_{2^{-n+1}}(u, v^i_{2^{-n}})\delta(u) = \sum_{\substack{u=(u_i)_{i=1}^d \in V_{n-1} \\ u_i= v_i +2^{-n}}} \Lambda^{d-1}_{2^{-n+1}}(\pi_i(u), \pi_i(v))\delta(u) \text,\\
        \sum_{u \in V_{n-1}} \Lambda^d_{2^{-n+1}}(u, v^i_{-2^{-n}})\delta(u) = \sum_{\substack{u=(u_i)_{i=1}^d \in V_{n-1} \\ u_i= v_i -2^{-n}}} \Lambda^{d-1}_{2^{-n+1}}(\pi_i(u), \pi_i(v))\delta(u) \text.
    \end{gather*}

    \pagebreak
    Altogether, we have
    \begin{multline*}
        2^{n\alpha}\left( \delta(v)-\frac 1 2 (\delta(v^i_{2^{-n}})+\delta(v^i_{-2^{-n}}))\right)\\
        \begin{aligned}
            &=2^{n\alpha}\left(\delta(v)-\sum_{u \in V_{n-1}}\Lambda^d_{2^{-n+1}}(u, v)\delta(u)\right)\\
            &\phantom{=}-2^{n\alpha-1}\left(\delta(v^i_{2^{-n}})-\sum_{u\in V_{n-1}}\Lambda^d_{2^{-n+1}}(u, v^i_{2^{-n}})\delta(u)\right)\\
            &\phantom{=}-2^{n\alpha-1}\left(\delta(v^i_{-2^{-n}})-\sum_{u \in V_{n-1}}\Lambda^d_{2^{-n+1}}(u, v^i_{-2^{-n}})\delta(u)\right)\text.
        \end{aligned} 
    \end{multline*}

    Note that by the assumption, $2^{n\alpha}\left(\delta(v^\prime)-\sum_{u \in V_{n-1}} \Lambda^d_{2^{-n+1}}(u, v^\prime)\delta(u)\right)$ for each $v^\prime \in\{v, v^i_{\pm 2^{-n}}\} \subseteq V_n$. The intermediate claim now follows as
    \begin{equation*}
        \begin{split}
            2^{n\alpha}\left( \delta(v)-\frac 1 2 (\delta(v^i_{2^{-n}})+\delta(v^i_{-2^{-n}}))\right)&\in \left(1+2^{1-p}\right)^{1/p}\aconv_p \mathcal M\\
            &\subseteq \rho \aconv_p \mathcal M \text.
        \end{split}
    \end{equation*}

    Let $l\in\{1, \ldots, d-1\}$ be such that the claim holds for $l-1$. Pick $v=(v_i)_{i=1}^d\in V$, $n\in\mathbb N$, and $i\in \{1, \ldots, d\}$ such that $v_i \in 2^{-n}\mathbb Z \setminus 2^{-n+1}\mathbb Z$. If $\{ j\in \{1, \ldots, d\} : v_j \notin 2^{-n}\mathbb Z\}$ is empty, the conclusion follows from the induction hypothesis; we may thus further assume there is some $j\in\{1, \ldots, d\}$ for which $v_j \notin 2^{-n}\mathbb Z$.
    
    We consider $u_1,\, u_2 \in [0, 1]\cap 2^{-n}\mathbb Z$ such that $u_1<v_j<u_2$ and $|u_1-u_2|=2^{-n}$. Let us further denote $v' = (v_1, \ldots, v_{j-1}, u_1, v_{j+1}, \ldots, v_d)$, $v'' = (v_1, \ldots, v_{j-1}, u_2, v_{j+1}, \ldots, v_d)$. Observe that $i$, $n$ and $v', \, v''\in V$, respectively, satisfy the induction hypothesis for $l-1$; hence,
    \begin{multline}\label{eq:step_inclusion_hypothesis}
        \left\{ 2^{n\alpha}\left( \delta(w)-\frac 1 2 (\delta(w^i_{2^{-n}})+\delta(w^i_{-2^{-n}}))\right) : w\in \{v', v''\} \right\}\\\subseteq \rho^l \aconv_p \mathcal M \text.
    \end{multline}

    We set $m=i$ if and only if $j>i$ and $m=i-1$ otherwise, and we define $x=(v_1, \ldots, v_{j-1}, v_{j+1}, \ldots, v_d)$ and $x_1=x^m_{2^{-n}}$, $x_2=x^m_{-2^{-n}}$. Let also $\mu_1, \mu_2 \in \mathbb R$, $l \in \mathbb N_0$, $\nu_r \in \mathbb R$, $n_r \in \mathbb N$, $n_r>n$, and $w_r \in 2^{-n_r}\mathbb Z \setminus 2^{-n_r+1}\mathbb Z$, where $r \in \{1, \ldots, l\}$, be the coefficients from~\Cref{lemma:hat} associated with $u_1$, $u_2$, and $v_j$. It follows that 
    \begin{align*}
        2^{n\alpha}\delta^j_x(v_j) &= \mu_1 2^{n\alpha}\delta^j_x(u_1)+\mu_2 2^{n\alpha}\delta^j_x(u_2)\\
            &\phantom{=}+\sum_{r=1}^l \nu_r 2^{\alpha n_r}\left(\delta^j_x (w_r) - \frac 1 2 (\delta^j_x (w_r^-)+\delta^j_x (w_r^+))\right) \text,\\
        2^{n\alpha-1}\delta^j_{x_1}(v_j) &= \mu_1 2^{n\alpha-1}\delta^j_{x_1}(u_1)+\mu_2 2^{n\alpha-1}\delta^j_{x_1}(u_2)\\
            &\phantom{=}+\sum_{r=1}^l \nu_r 2^{\alpha n_r-1}\left(\delta^j_{x_1} (w_r) - \frac 1 2 (\delta^j_{x_1} (w_r^-)+\delta^j_{x_1} (w_r^+))\right) \text,\\
        \intertext{and}
        2^{n\alpha-1}\delta^j_{x_2}(v_j) &= \mu_1 2^{n\alpha-1}\delta^j_{x_2}(u_1)+\mu_2 2^{n\alpha-1}\delta^j_{x_2}(u_2)\\
            &\phantom{=}+\sum_{r=1}^l \nu_r 2^{\alpha n_r-1}\left(\delta^j_{x_2} (w_r) - \frac 1 2 (\delta^j_{x_2} (w_r^-)+\delta^j_{x_2} (w_r^+))\right) \text.
    \end{align*}

    Pick $r\in \{1, \ldots, l\}$ and recall that $w_r \in 2^{-n_r}\mathbb Z\setminus 2^{-n}\mathbb Z$. If $y=(y_i)_{i=1}^{d-1}\in\{x, x_1, x_2\}$, we set $y^\prime = (y^\prime_i)_{i=1}^{d}=(y_1, \ldots, y_{j-1}, w_r, y_{j+1}, \ldots, y_d)$; it follows that $|\{ k\in \{1, \ldots, d\} : y^\prime_k \notin 2^{-n_r}\mathbb Z\}| \leq l-1$. By the induction hypothesis, \[ 2^{\alpha n_r}\left(\delta^j_{y} (w_r) - \frac 1 2 (\delta^j_{y} (w_r^-)+\delta^j_{y} (w_r^+))\right) \in \rho^l \aconv_p \mathcal M \text.\]
    
    Using the estimate from~\Cref{lemma:hat}, we obtain for any $y\in \{x, x_1, x_2\}$
    \begin{multline}\label{eq:step_inclusion_hat}
        \sum_{r=1}^l \nu_r 2^{\alpha n_r}\left(\delta^j_{y} (w_r) - \frac 1 2 (\delta^j_{y} (w_r^-)+\delta^j_{y} (w_r^+))\right) \\\in 2^{-\alpha}\left(\frac 1 {1-2^{-p\alpha}}\right)^{1/p}\rho^l \aconv_p \mathcal M \text.
    \end{multline}

    We note that $\delta^j_x(v_j) = \delta(v)$; similarly, 
    \begin{align*}
        \delta^j_x(u_1) &= \delta(v') & \delta^j_x(u_2) &= \delta(v'')\\
        \delta(v^i_{2^{-n}}) &=\delta^j_{x_1}(v_j) & \delta(v^i_{-2^{-n}}) &=\delta^j_{x_2}(v_j)\\
        \delta^j_{x_1}(u_1)&=\delta((v')^i_{2^{-n}}) & \delta^j_{x_1}(u_2)&=\delta((v'')^i_{2^{-n}})\\
        \delta^j_{x_2}(u_1)&=\delta((v')^i_{-2^{-n}}) & \delta^j_{x_2}(u_2)&=\delta((v'')^i_{-2^{-n}}) \text.
    \end{align*}
    
    We may now rewrite
    \begin{multline*}
        2^{n\alpha}\left( \delta(v)-\frac 1 2 (\delta(v^i_{2^{-n}})+\delta(v^i_{-2^{-n}}))\right) \\
        \begin{aligned}
            &= \mu_1 2^{n\alpha}\delta^j_x(u_1)+\mu_2 2^{n\alpha}\delta^j_x(u_2)\\
            &\phantom{=}+\sum_{r=1}^l \nu_r 2^{\alpha n_r}\left(\delta^j_x (w_r) - \frac 1 2 (\delta^j_x (w_r^-)+\delta^j_x (w_r^+))\right)\\
            &\phantom{=}-\mu_1 2^{n\alpha-1}\delta^j_{x_1}(u_1)+\mu_2 2^{n\alpha-1}\delta^j_{x_1}(u_2)\\
            &\phantom{=}-\sum_{r=1}^l \nu_r 2^{\alpha n_r-1}\left(\delta^j_{x_1} (w_r) - \frac 1 2 (\delta^j_{x_1} (w_r^-)+\delta^j_{x_1} (w_r^+))\right)\\
            &\phantom{=}-\mu_1 2^{n\alpha-1}\delta^j_{x_2}(u_1)+\mu_2 2^{n\alpha-1}\delta^j_{x_2}(u_2)\\
            &\phantom{=}-\sum_{r=1}^l \nu_r 2^{\alpha n_r-1}\left(\delta^j_{x_2} (w_r) - \frac 1 2 (\delta^j_{x_2} (w_r^-)+\delta^j_{x_1} (w_r^+))\right) \text,
        \end{aligned}
    \end{multline*}
    and
    \begin{multline*}
            \mu_1 2^{n\alpha}\left(\delta^j_x(u_1)-\frac 1 2 (\delta^j_{x_1}(u_1)+\delta^j_{x_2}(u_1))\right) \\
            = \mu_1 2^{n\alpha}\left( \delta(v')-\frac 1 2 (\delta((v')^i_{2^{-n}})+\delta((v')^i_{-2^{-n}}))\right) \text,
    \end{multline*}
    \begin{multline*}
            \mu_2 2^{n\alpha}\left(\delta^j_x(u_2)-\frac 1 2 (\delta^j_{x_1}(u_2)+\delta^j_{x_2}(u_2))\right) \\
        = \mu_2 2^{n\alpha}\left(\delta(v'')-\frac 1 2 (\delta((v'')^i_{2^{-n}})+\delta((v'')^i_{-2^{-n}}))\right) \text.
    \end{multline*}

    Substituting the last two terms, let us rewrite
    \begin{multline*}
        2^{n\alpha}\left( \delta(v)-\frac 1 2 (\delta(v^i_{2^{-n}})+\delta(v^i_{-2^{-n}}))\right)\\
        \begin{aligned}
        &=\mu_1 2^{n\alpha}\left(\delta(v')-\frac 1 2 (\delta((v')^i_{2^{-n}})+\delta((v')^i_{-2^{-n}}))\right) \\
        &\phantom{=}+\mu_2 2^{n\alpha}\left(\delta(v'')-\frac 1 2 (\delta((v'')^i_{2^{-n}})+\delta((v'')^i_{-2^{-n}}))\right)\\
        &\phantom{=}+\sum_{r=1}^l \nu_r 2^{\alpha n_r}\left(\delta^j_x (w_r) - \frac 1 2 (\delta^j_x (w_r^-)+\delta^j_x (w_r^+))\right)\\
        &\phantom{=}-\frac 1 2\sum_{r=1}^l \nu_r 2^{\alpha n_r}\left(\delta^j_{x_1} (w_r) - \frac 1 2 (\delta^j_{x_1} (w_r^-)+\delta^j_{x_1} (w_r^+))\right)\\
        &\phantom{=}-\frac 1 2\sum_{r=1}^l \nu_r 2^{\alpha n_r}\left(\delta^j_{x_2} (w_r) - \frac 1 2 (\delta^j_{x_2} (w_r^-)+\delta^j_{x_1} (w_r^+))\right) \text.
        \end{aligned}
    \end{multline*}

    Appealing to~\labelcref{eq:step_inclusion_hypothesis,eq:step_inclusion_hat}, we conclude 
    \begin{multline*}
        2^{n\alpha}\left( \delta(v)-\frac 1 2 (\delta(v^i_{2^{-n}})+\delta(v^i_{-2^{-n}}))\right)\\
        \begin{aligned}
            &\in \left(\mu_1^p + \mu_2^p + (1+2^{1-p}) 2^{-p\alpha} \frac 1 {1-2^{-p\alpha}}\right)^{1/p} \rho^{l} \aconv_p \mathcal M \\
            &\subseteq \rho^{l+1} \aconv_p \mathcal M \text.
        \end{aligned}
    \end{multline*}

    This verifies the induction step, and thus completes the proof.
\end{proof}

\subsubsection{One~Geometrical Consideration}

We develop a~result which, for $d=1$ in particular, implies there exists $\rho\in\mathbb R$ such that for any $u,\,v \in [0, 1] \cap 2^{-n}\mathbb Z$, where $n\in\mathbb N_0$, we get $\frac{\delta(u)-\delta(v)}{|u-v|^\alpha} \in \rho\aconv_p \mathcal \iota(\{ e_x : x\in V\})$. 

First we establish the conclusion under the assumption that $|u-v|=2^{-n}$.

\begin{lemma} \label{lemma:line}
    Let $d\in\mathbb N$, $i\in\{1, \ldots, d\}$, $x=(x_j)_{j=1}^{d-1}\in [0, 1]^{d-1}$, and $\mathcal M \subseteq \mathcal F_p ([0, 1]^d, |\cdot|^\alpha)$. Assume that $\delta^i_x (0),\, \delta^i_x (1) \in \mathcal M$ and for any $y\in 2^{-k}\mathbb Z \setminus 2^{-k+1}\mathbb Z$, where $k\in\mathbb N$, it holds that $2^{\alpha k}\left(\delta^i_x (y) - \frac 1 2 (\delta^i_x (y^-)+\delta^i_x (y^+))\right) \in \mathcal M$. 
    
    Whenever $u,\,v \in [0, 1] \cap 2^{-n}\mathbb Z$ and $|u-v|=2^{-n}$ for some $n\in\mathbb N_0$, we have $\frac{\delta^i_x(u)-\delta^i_x(v)}{|u-v|^\alpha} \in 2^{1/p}\left(\frac 1 {1-2^{p(\alpha-1)}}\right)^{1/p}\aconv_p \mathcal M$.
\end{lemma}

\begin{proof}
    Given $n \in \mathbb N_0$, let us denote \[\mathcal M_n = \left\{ \frac{\delta^i_x(u)-\delta^i_x(v)}{|u-v|^\alpha} : u,\,v \in [0, 1] \cap 2^{-n}\mathbb Z,\,|u-v|=2^{-n}\right\} \text.\]
    
    We claim that for any $y \in \mathcal M_n$, where $n\in\mathbb N_0$, there exist coefficients $\mu^y_j \in \{-1, 1\}$, where $j\in \{0, \ldots, n\}$, and elements $y^y_j \in 2^{-j}\mathbb Z \setminus 2^{-j+1}\mathbb Z$, where $j\in\{1, \ldots, n\}$, such that     
    \begin{equation} \label{eq:line_expansion}
        \begin{split}
            y&=\mu^y_0 2^{n(\alpha-1)}(\delta^i_x (1)-\delta^i_x (0))\\
            &\phantom{=} + \sum_{j=1}^n \mu^y_j 2^{(n-j)(\alpha-1)}2^{j\alpha}\left(\delta^i_x(y^y_j) - \frac 1 2 (\delta^i_x((y^y_j)^+)+\delta^i_x((y^y_j)^-))\right).
        \end{split}
    \end{equation}

    We proceed by induction on $n$. To that end, we note the conclusion is trivial in the case $n=0$.

    Let $n\in \mathbb N$ be such that the claim holds for $n-1$, and pick $y = \frac{\delta^i_x(u)-\delta^i_x(v)}{|u-v|^\alpha}\in \mathcal M_n$ for some $u,\,v \in [0, 1] \cap 2^{-n}\mathbb Z$, $|u-v|=2^{-n}$. As $\{u, v\} \cap 2^{-n}\mathbb Z \setminus 2^{-n+1}\mathbb Z \neq \emptyset$, there exists $w \in 2^{-n}\mathbb Z \setminus 2^{-n+1}\mathbb Z$ for which \[y \in \left\{ \pm\frac{\delta^i_x(w)-\delta^i_x(w^+)}{|w-w^+|^\alpha}, \pm\frac{\delta^i_x(w)-\delta^i_x(w^-)}{|w-w^-|^\alpha} \right\} \text.\] 
    
    Similarly, there exist $\nu_1,\,\nu_2\in\{-1, 1\}$ such that 
    \begin{equation}\label{eq:line_combination}
        \begin{split}
            \frac{\delta^i_x(u)-\delta^i_x(v)}{|u-v|^\alpha} &= \nu_1 2^{n\alpha}\left(\delta^i_x(w) - \frac 1 2 (\delta^i_x(w^+)+\delta^i_x(w^-))\right)\\
            &\phantom{=} + \nu_2 2^{\alpha-1}\frac{\delta^i_x(w^+)-\delta^i_x(w^-)}{|w^+-w^-|^\alpha}.
        \end{split}
    \end{equation}

    Denote $z=\frac{\delta^i_x(w^+)-\delta^i_x(w^-)}{|w^+-w^-|^\alpha}\in\mathcal M_{n-1}$ and let $\mu^z_j \in \{-1, 1\}$, where $j\in \{0, \ldots, n-1\}$, and $y^z_j \in 2^{-j}\mathbb Z \setminus 2^{-j+1}\mathbb Z$, where $j\in\{1, \ldots, n-1\}$, be the coefficients from the inductive hypothesis. We may now define $\mu^y_j=\nu_2\mu^z_j$, where $j\in \{0, \ldots, n-1\}$, and $y^y_j=y^z_j$, where $j\in \{1, \ldots, n-1\}$. Similarly, let $\mu^y_n=\nu_1$, $y^y_n=w$.
    
    It follows from the construction that $\{\mu^y_j : j\in \{0, \ldots, n\}\} \subseteq \{-1, 1\}$ and $y^y_j \in 2^{-j}\mathbb Z \setminus 2^{-j+1}\mathbb Z$ for any $j\in\{1, \ldots, n\}$. By~\eqref{eq:line_combination} we easily verify
    \[ 
        \begin{split}
            y&=\mu^y_0 2^{n(\alpha-1)}(\delta^i_x (1)-\delta^i_x (0))\\
            &\phantom{=} + \sum_{j=1}^n \mu^y_j 2^{(n-j)(\alpha-1)}2^{j\alpha}\left(\delta^i_x(y^y_j) - \frac 1 2 (\delta^i_x((y^y_j)^+)+\delta^i_x((y^y_j)^-))\right)\text.
        \end{split}
    \]
    
    As $y\in\mathcal M_n$ was arbitrary, this concludes the induction step.

    Let $y\in\mathcal M_n$, where $n\in\mathbb N_0$. We pick $\mu^y_j \in \{-1, 1\}$, where $j\in \{0, \ldots, n\}$, and $y^y_j \in 2^{-j}\mathbb Z \setminus 2^{-j+1}\mathbb Z$, where $j\in\{1, \ldots, n\}$, as found in the previous part. Recall that by the assumption, $2^{j\alpha}\left(\delta^i_x(y^y_j) - \frac 1 2 (\delta^i_x((y^y_j)^+)+\delta^i_x((y^y_j)^-))\right) \in \mathcal M$ for any $j\in\{1, \ldots, n\}$ and $\delta^i_x (0),\, \delta^i_x (1) \in \mathcal M$. Combining this with~\eqref{eq:line_expansion}, we deduce $y\in \left( 2\cdot2^{p(\alpha-1)n}+\sum_{j=1}^n 2^{p(n-j)(\alpha-1)}\right)^{1/p}\aconv_p \mathcal M$; hence,
    \[ y\in \left( 2\sum_{i=0}^n 2^{pi(\alpha-1)}\right)^{1/p}\aconv_p \mathcal M \subseteq 2^{1/p}\left(\frac 1 {1-2^{p(\alpha-1)}}\right)^{1/p}\aconv_p \mathcal M \text.\] The claim follows.
\end{proof}

We decompose a~general element $\frac{\delta(u)-\delta(v)}{|u-v|^\alpha}$, where $u,\,v \in [0, 1] \cap 2^{-n}\mathbb Z$ for some $n\in\mathbb N_0$, into a~combination of elements which we considered above.

The following technical result is a~refinement of~\cite[Lemma 8.40]{Weaver2018}.

\begin{lemma}\label{lemma:path}
    Let $u,\,v \in [0, 1] \cap 2^{-n}\mathbb Z$ for some $n\in \mathbb N_0$, $u\neq v$. There exist $l\in\mathbb N$ and $a_i \in [u, v] \cap 2^{-n} \mathbb Z$, where $i\in\{1, \ldots, l\}$, such that
    \begin{enumerate}[label=(\roman*), ref=\roman*]
        \item\label{it:path_1} $a_1=u$, $a_l=v$,
        \item\label{it:path_2} $a_i,\, a_{i+1} \in 2^{-k}\mathbb Z$ and $|a_i-a_{i+1}|=2^{-k}$ for some $k \in \{0, \ldots, n\}$, for any $i\in \{1, \ldots, l-1\}$,
        \item\label{it:path_3} $\left( \sum_{i=1}^{l-1} |a_{i+1}-a_i|^{p\alpha} \right)^{1/p} < 2^{1/p}\left(\frac 1 {1-2^{-p\alpha}}\right)^{1/p}|u-v|^\alpha$.
    \end{enumerate}
\end{lemma}

\begin{proof}
    Without loss of generaliy, let $u<v$.
    
    We pick the least $n_0 \in \mathbb N_0 \cup \{-1\}$ for which $[u, v] \cap 2^{-n_0}\mathbb Z \neq \emptyset$ and define the sets $\mathcal V_i \subseteq  2^{-i} \mathbb Z$, where $i\in \{n_0, \ldots, n\}$, as follows. Set $\mathcal V_{n_0} = 2^{-n_0}\mathbb Z \cap [u, v]$. If $i\in \{n_0+1, \ldots, n\}$ and $\mathcal V_{j}$ was defined for all $j\in \{n_0, \ldots, i-1\}$, we take $\mathcal V_i = [u, v] \cap 2^{-i}\mathbb Z \setminus [\min \bigcup_{j\in \{n_0, \ldots, i-1\}}\mathcal V_{j}, \max \bigcup_{j\in \{n_0, \ldots, i-1\}}\mathcal V_{j}]$. 

    We remark that $|\mathcal V_{n_0}|=1$ by the choice of $n_0$ and, moreover, $[u, v] \cap 2^{-i}\mathbb Z \subseteq [\min \bigcup_{j\in \{n_0, \ldots, i\}}\mathcal V_{j}, \max \bigcup_{j\in \{n_0, \ldots, i\}}\mathcal V_{j}]$, where $i\in \{n_0, \ldots, n\}$, by the construction. 
    
    Pick $i\in \{n_0+1, \ldots, n\}$. It follows from the above remark that $\mathcal V_i \cap 2^{-i+1}\mathbb Z = \emptyset$; since $\bigcup_{j\in \{n_0, \ldots, i-1\}}\mathcal V_{j} \subseteq 2^{-i+1}\mathbb Z$, this establishes the inclusion
    \begin{equation}\label{eq:path_expansion}
        \mathcal V_i \subseteq \{\min \bigcup\nolimits_{j\in \{n_0, \ldots, i-1\}}\mathcal V_{j} - 2^{-i}, \max \bigcup\nolimits_{j\in \{n_0, \ldots, i-1\}}\mathcal V_{j}+2^{-i}\} \text.
    \end{equation}

    Denote $\mathcal V = \bigcup_{i\in \{n_0, \ldots, n\}} \mathcal V_i$. It follows that $\mathcal V$ is finite and $\{u, v\} \subseteq \mathcal V$ since $\{u, v\} \subseteq 2^{-n}\mathbb Z$; we may thus define $l=|\mathcal V|$ and find a strictly monotone arrangement $a_i$, where $i\in\{1, \ldots, l\}$, of the set $\mathcal V$ such that $a_1=u$, $a_l=v$. It remains to show that $l$ and $a_i$, where $i\in\{1, \ldots, l\}$, satisfy~\cref{it:path_2,it:path_3}.
    
    To show~\labelcref{it:path_2}, pick $i\in \{1, \ldots, l-1\}$. By~\eqref{eq:path_expansion}, there is $k\in \{n_0+1, \ldots, n\}$ for which $a_i=\min \bigcup_{j\in \{n_0, \ldots, k-1\}}\mathcal V_{j} - 2^{-k}$, $a_{i+1}=\min \bigcup_{j\in \{n_0, \ldots, k-1\}} \mathcal V_j$ or $a_i=\max \bigcup_{j\in \{n_0, \ldots, k-1\}}\mathcal V_{j}$, $a_{i+1}=\min \bigcup_{j\in \{n_0, \ldots, k-1\}}\mathcal V_{j} + 2^{-k}$; it follows that $a_i,\,a_{i+1}\in 2^{-k}\mathbb Z$ and $|a_i-a_{i+1}|=2^{-k}$.
    
    The above argument also shows that the set $\{i\in \{1, \ldots, l-1\}: |a_i-a_{i+1}|=2^{-k}\}$ is empty for any $k\in \mathbb Z$ such that $k\leq n_0$ or $n<k$ and it has at most two elements if $n_0<k\leq n$. We consider the least $k_0\in\mathbb Z$ for which there exists $i\in \{1, \ldots, l-1\}$ with $|a_i-a_{i+1}|= 2^{-k_0}$; since $|u-v|\geq 2^{-k_0}$, we have
    \[
        \begin{split}
            \left( \sum_{i=1}^{l-1} \frac{|a_{i+1}-a_i|^{p\alpha}}{|u-v|^{p\alpha}} \right)^{1/p} &\leq \left( \sum_{i=1}^{l-1} \frac{|a_{i+1}-a_i|^{p\alpha}}{2^{-pk_0\alpha}} \right)^{1/p}\\
            &\leq \left( 2 \sum_{i=k_0}^{n} 2^{p(-i+k_0)\alpha} \right)^{1/p}\\ 
            &<2^{1/p}\left(\frac 1 {1-2^{-p\alpha}}\right)^{1/p}\text,
        \end{split} 
    \]
    which verifies~\cref{it:path_3}. The proof is now complete.
\end{proof}

\begin{lemma} \label{lemma:line_general}
    Let $d\in\mathbb N$, $i\in\{1, \ldots, d\}$, $x=(x_j)_{j=1}^{d-1}\in [0, 1]^{d-1}$, and $\mathcal M \subseteq \mathcal F_p ([0, 1]^d, |\cdot|^\alpha)$ be such that the conclusion of~\Cref{lemma:line} holds true, i.e., whenever $u,\,v \in [0, 1] \cap 2^{-n}\mathbb Z$ and $|u-v|=2^{-n}$ for some $n\in\mathbb N_0$, then $\frac{\delta^i_x(u)-\delta^i_x(v)}{|u-v|^\alpha} \in 2^{1/p}\left(\frac 1 {1-2^{p(\alpha-1)}}\right)^{1/p}\aconv_p \mathcal M$.
    
    Whenever $u,\,v \in [0, 1] \cap 2^{-n}\mathbb Z$ for some $n\in\mathbb N_0$, where $u\neq v$, then $\frac{\delta^i_x(u)-\delta^i_x(v)}{|u-v|^\alpha} \in 2^{2/p}\left(\frac 1 {1-2^{p(\alpha-1)}}\right)^{1/p}\left(\frac 1 {1-2^{-p\alpha}}\right)^{1/p}\aconv_p \mathcal M$.
\end{lemma}

\begin{proof}
    Given $u,\,v \in [0, 1] \cap 2^{-n}\mathbb Z$, where $n\in\mathbb N$, $u\neq v$, we pick the coefficients $l$ and $a_j$, where $j\in\{1,\ldots, l\}$, from~\Cref{lemma:path} associated with $v$, $u$, and $n$.

    By~\cref{it:path_1,it:path_3} of~\Cref{lemma:path}, respectively, we have
    \begin{gather*}
        \frac{\delta^i_x(u)-\delta^i_x(v)}{|u-v|^\alpha} = \sum_{j=1}^{l-1} \frac{|a_{j+1}-a_j|^\alpha}{|u-v|^\alpha}\frac{\delta^i_x(a_{j+1})-\delta^i_x(a_j)}{|a_{j+1}-a_j|^\alpha},\\
        \left( \sum_{j=1}^{l-1} \frac{|a_{j+1}-a_j|^{p\alpha}}{|u-v|^{p\alpha}} \right)^{1/p} < 2^{1/p}\left(\frac 1 {1-2^{-p\alpha}}\right)^{1/p} \text.
    \end{gather*}
    By the assumption on $\mathcal M$ and~\cref{it:path_2}, wes obtain
    \[ \frac{\delta^i_x(a_{j+1})-\delta^i_x(a_j)}{|a_{j+1}-a_j|^\alpha}\in 2^{1/p}\left(\frac 1 {1-2^{p(\alpha-1)}}\right)^{1/p}\aconv_p \mathcal M, \quad j\in\{1, \ldots, l-1\} \text; \]
    hence, we get $\frac{\delta^i_x(u)-\delta^i_x(v)}{|u-v|^\alpha} \in 2^{2/p}\left(\frac 1 {1-2^{p(\alpha-1)}}\right)^{1/p}\left(\frac 1 {1-2^{-p\alpha}}\right)^{1/p}\aconv_p \mathcal M$. The proof is complete.
\end{proof}
\smallskip

Drawing on the conclusion of~\Cref{lemma:step}, we generalize the previous result by induction on the dimension of boundary cubes.

\begin{lemma} \label{lemma:combinations}
    Let $d\in\mathbb N$ and $\mathcal M \subseteq \mathcal F_p ([0, 1]^d, |\cdot|^\alpha)$ be such that $\{ \delta(u) : u\in\{0, 1\}^d \} \subseteq \mathcal M$ and the conclusion of~\Cref{lemma:step} holds true, i.e., there exists $\rho^\prime>0$ such that for any $v=(v_i)_{i=1}^d\in V$ and $i\in \{1, \ldots, d\}$, if $v_i \in 2^{-n}\mathbb Z \setminus 2^{-n+1}\mathbb Z$ for some $n\in\mathbb N$, then $2^{n\alpha}\left( \delta(v)-\frac 1 2 (\delta(v^i_{2^{-n}})+\delta(v^i_{-2^{-n}}))\right) \in \rho^\prime\aconv_p \mathcal M$.
    
    There exists $\tau^\prime>0$ such that for any $u,\,v \in V \cup \{ 0 \}$, $u\neq v$, we have $\frac{\delta(u)-\delta(v)}{|u-v|^\alpha} \in \tau^\prime\aconv_p \mathcal M$. Quantitatively, if we denote $\tau=C^\alpha(p\alpha, d) 2^{2/p} \cdot \left(\frac 1 {1-2^{p(\alpha-1)}}\right)^{1/p} \left(\frac 1 {1-2^{-p\alpha}}\right)^{1/p} \cdot (1+(d-1)^{p\alpha})^{1/p}$, then $\tau' = \tau^d \rho'$.
\end{lemma}

\begin{proof}
    We inductively show that for any $l\in\{1, \ldots, d\}$, the following claim is true. Let $u=(u_i)_{i=1}^d,\,v=(v_i)_{i=1}^d\in [0, 1]^d \cap 2^{-n}\mathbb Z^d$, where $n\in\mathbb N_0$, $u\neq v$, and $\mathcal I \subseteq \{1, \ldots, d\}$, $|\mathcal I|=d-l$, be such that $u_i=v_i\in\{0, 1\}$ for any $i\in \mathcal I$. Then $\frac{\delta(u)-\delta(v)}{|u-v|^\alpha} \in \tau^l\rho'\aconv_p \mathcal M$.

    Let us remark that by assuming $u_i=v_i\in\{0, 1\}$ for any $i\in \mathcal I$, where $|\mathcal I|=d-l$, we iteratively develop the result over $l$-faces of~$[0, 1]^d$.
    
    Let $l=1$ and pick $u=(u_i)_{i=1}^d,\, v=(v_i)_{i=1}^d\in [0, 1]^d\cap 2^{-n}\mathbb Z^d$, where $n\in\mathbb N_0$, $u\neq v$, and $i\in\{1, \ldots, d\}$, such that $u_j=v_j\in\{0, 1\}$ for any $j\in\{1, \ldots, d\} \setminus \{i\}$. We may find $x\in\{0, 1\}^{d-1}$ and $u^\prime, \, v^\prime \in [0, 1]\cap 2^{-n}\mathbb Z$ such that $\delta(u)=\delta^i_x (u^\prime)$, $\delta(v)=\delta^i_x (v^\prime)$.
    
    Note that by the assumption, $\delta^i_x (0),\, \delta^i_x (1) \in \mathcal M$, and whenever $q\in 2^{-k}\mathbb Z \setminus 2^{-k+1}\mathbb Z$ for some $k\in\mathbb N$, we have $2^{\alpha k}\left(\delta^i_x (q) - \frac 1 2 (\delta^i_x (q^-)+\delta^i_x (q^+))\right) \in \rho'\aconv_p\mathcal M$. It follows that $d$, $i$, $x$, and $\rho'\aconv_p\mathcal M$ satisfy the assumption of~\Cref{lemma:line}. 
    
    Consequently, by~\Cref{lemma:line_general} for $u^\prime,\, v^\prime\in [0, 1] \cap 2^{-n}\mathbb Z$, $u\neq v$, we obtain \[\frac{\delta(u)-\delta(v)}{|u-v|^\alpha} = \frac{\delta^i_x (u^\prime)-\delta^i_x (v^\prime)}{|u^\prime-v^\prime|^\alpha} \in \tau\rho'\aconv_p \mathcal M \text,\] which establishes the first step.

    Let $l\in \{2, \ldots, d\}$ be such that the claim holds for $l-1$. We pick $u=(u_i)_{i=1}^d,\, v=(v_i)_{i=1}^d\in [0, 1]^d \cap 2^{-n}\mathbb Z^d$, where $n\in\mathbb N_0$, $u\neq v$, and $\mathcal I\subseteq \{1, \ldots, d\}$, $|\mathcal I|=d-l$, such that $u_i=v_i\in\{0, 1\}$ for any $i\in \mathcal I$. Assume first there exists $i\in\{1, \ldots, d\}\setminus \mathcal I$ such that $u_j=v_j$ for any $j\in \{1, \ldots, d\} \setminus \{i\}$. We shall prove that $\frac{\delta(u)-\delta(v)}{|u-v|^\alpha} \in 2^{2/p}\left(\frac 1 {1-2^{p(\alpha-1)}}\right)^{1/p}\left(\frac 1 {1-2^{-p\alpha}}\right)^{1/p}(1+(d-1)^{p\alpha})^{1/p}\tau^{l-1}\rho'\aconv_p \mathcal M$.

    Let $y=(y_j)_{j=1}^d, \, z=(z_j)_{j=1}^d\in [0, 1]^d \cap 2^{-n}\mathbb Z^d$ be such that $y_i=0$, $z_i=1$ and $y_j=z_j=u_j$ for any $j\in\{1, \ldots, d\}\setminus \{i\}$. We further pick $y^\prime=(y^\prime_j)_{j=1}^d, \, z^\prime=(z^\prime_j)_{j=1}^d \in \{0, 1\}^d$ satisfying $y^\prime_j=y_j$ and $z^\prime_j=z_j$ for any $j\in\mathcal I \cup \{i\}$.
    
    If $y\neq y^\prime$, we rewrite $\delta(y) = |y-y'|^\alpha\frac{\delta (y)-\delta(y^\prime)}{|y-y^\prime|^\alpha} + \delta(y')$, where, in particular, $\frac{\delta (y)-\delta(y^\prime)}{|y-y^\prime|^\alpha} \in \tau^{l-1}\rho'\aconv_p \mathcal M$ by the induction hypothesis. Since $\delta(y) = \delta(y')$ in the remaining case and $\delta(y^\prime) \in \mathcal M$ by the assumption on $\mathcal M$, we deduce that $\delta(y) \in (1+(d-1)^{p\alpha})^{1/p}\tau^{l-1}\rho'\aconv_p \mathcal M$. Repeating the same argument for $z$, we conclude $\delta(y),\,\delta(z) \in (1+(d-1)^{p\alpha})^{1/p}\tau^{l-1}\rho'\aconv_p \mathcal M$.
    
    Let next $x\in[0,1]^{d-1}\cap 2^{-n}\mathbb Z^{d-1}$ and $u^\prime, \, v^\prime \in [0, 1]\cap 2^{-n}\mathbb Z$, $u'\neq v'$, be such that $\delta(u)=\delta^i_x (u^\prime)$, $\delta(v)=\delta^i_x (v^\prime)$. We note it holds $\delta(y)=\delta^i_x (0)$ and $\delta(z)=\delta^i_x (1)$. By the preceding paragraph, we have $\delta^i_x (0),\, \delta^i_x (1) \in (1+(d-1)^{p\alpha})^{1/p}\tau^{l-1}\rho'\aconv_p \mathcal M$, and if $q\in 2^{-k}\mathbb Z \setminus 2^{-k+1}\mathbb Z$ for some $k\in\mathbb N$, then $2^{\alpha k}\left(\delta^i_x (q) - \frac 1 2 (\delta^i_x (q^-)+\delta^i_x (q^+))\right) \in \rho'\aconv_p \mathcal M$ by the assumption on $\mathcal M$. By~\Cref{lemma:line,lemma:line_general}, we conclude 
    \[
        \begin{split}
            \frac{\delta(u)-\delta(v)}{|u-v|^\alpha} &= \frac{\delta^i_x (u^\prime)-\delta^i_x (v^\prime)}{|u^\prime-v^\prime|^\alpha}\\ &\in 2^{2/p}\left(\frac 1 {1-2^{p(\alpha-1)}}\right)^{1/p}\left(\frac 1 {1-2^{-p\alpha}}\right)^{1/p}\\
            &\phantom{=}\cdot(1+(d-1)^{p\alpha})^{1/p}\tau^{l-1}\rho'\aconv_p \mathcal M \text.
        \end{split}
    \]
    
    To establish the general case, we pick $u=(u_i)_{i=1}^d, \, v=(v_i)_{i=1}^d\in [0, 1]^d \cap 2^{-n}\mathbb Z^d$, where $n\in\mathbb N_0$, $u\neq v$, and $\mathcal I\subseteq \{1, \ldots, d\}$, $|\mathcal I|=d-l$, are such that $u_i=v_i\in\{0, 1\}$ for any $i\in \mathcal I$. If $u$, $v$ differ at exactly $k$ coordinates, $k\geq 1$, we set $u^1=v$, $u^{k+1}=u$ and find $u^i$, where $i\in\{2, \ldots, k\}$, such that $u^i$, $u^{i+1}$ differ at one coordinate for any $i\in\{1, \ldots, k\}$. Let us remark that all $u^i$, where $i\in\{1, \ldots, k+1\}$, mutually coincide at any coordinate from the set $\mathcal I$.
    
    It follows that $\sum_{i=1}^{k} |u^{i+1}-u^i| = |u-v|$ and
    \begin{gather*}
        \frac{\delta(u)-\delta(v)}{|u-v|^\alpha} = \sum_{i=1}^{k} \frac{|u^{i+1}-u^i|^\alpha}{|u-v|^\alpha}\frac{\delta(u^{i+1})-\delta(u^i)}{|u^{i+1}-u^i|^\alpha},\\
        \left(\sum_{i=1}^{k} \frac{|u^{i+1}-u^i|^{p\alpha}}{|u-v|^{p\alpha}}\right)^{1/p}\leq C^\alpha(p\alpha, d)\text,
    \end{gather*}
    where, by the already proven that, we have for any $i\in\{1, \ldots, k\}$, 
    \begin{equation*}
        \begin{split}
            \frac{\delta(u^{i+1})-\delta(u^i)}{|u^{i+1}-u^i|^\alpha} &\in 2^{2/p}\left(\frac 1 {1-2^{p(\alpha-1)}}\right)^{1/p}\left(\frac 1 {1-2^{-p\alpha}}\right)^{1/p}\\
            &\phantom{\in}\cdot(1+(d-1)^{p\alpha})^{1/p}\tau^{l-1}\rho'\aconv_p \mathcal M \text.
        \end{split}
    \end{equation*}
    
    Hence, we are justified to write 
    \[
        \begin{split}
            \frac{\delta(u)-\delta(v)}{|u-v|^\alpha} &\in C^\alpha(p\alpha, d)2^{2/p}\left(\frac 1 {1-2^{p(\alpha-1)}}\right)^{1/p}\left(\frac 1 {1-2^{-p\alpha}}\right)^{1/p}\\
            &\phantom{\in}\cdot(1+(d-1)^{p\alpha})^{1/p}\tau^{l-1}\rho'\aconv_p \mathcal M\\ 
            &= \tau^{l}\rho'\aconv_p \mathcal M.
        \end{split}
    \]
    The claim follows.
\end{proof}
\smallskip

By now we have collected all the necessary results to establish the isomorphism theorem.

\begin{proof}[Proof of~\Cref{thm:isomorphism_p_cube}.]
    Let $\iota : \{ e_v \}_{v\in V} \subset \ell_p (V) \to \mathcal F_p ([0, 1]^d, |\cdot|^\alpha)$ be defined as $e_v \mapsto 2^{k\alpha}\left(\delta(v)-\sum_{u \in V_{k-1}} \Lambda^d_{2^{-k+1}}(u, v)\delta(u)\right)$ for $v\in V_k\setminus V_{k-1}$, where $k\in\mathbb N_0$.

    It follows easily that $\{ e_v : v\in V \}$ is isometrically $p$-norming in $\ell_p (V)$ and that $\aconv_p \{ e_v : v\in V \}$ contains a~neighborhood of zero in $\operatorname{span} \{ e_v : v\in V \}$. Moreover, we claim that $\iota$ extends to an~one-to-one linear map from $\operatorname{span} \{ e_v : v\in V\}$ into $\mathcal F_p ([0, 1]^d, |\cdot|^\alpha)$. Indeed, an~easy direct argument is possible or it suffices to note that $\delta_{\mathcal F_p ([0, 1]^d, |\cdot|^\alpha)}(x) \mapsto \delta_{\mathcal F_p ([0, 1]^d)} (x)$, where $x\in [0, 1]^d$, induces an~onto linear bijection $\kappa : \operatorname{span}\{ \delta_{\mathcal F_p ([0, 1]^d, |\cdot|^\alpha)}(x) : x \in [0, 1]^d\} \to \operatorname{span}\{ \delta_{\mathcal F_p ([0, 1]^d)}(x) : x \in [0, 1]^d\}$ and $\kappa(\iota(e_v))$, where $v\in V$, form a~Schauder basis in $\mathcal F_p ([0, 1]^d)$, see~\cite[Theorem~3.8]{Albiac2022}.
    
    Once we show that $\iota(\{e_v : v\in V\})$ is $p$-norming in $\mathcal F_p ([0, 1]^d, |\cdot|^\alpha)$ and that $\aconv_p \iota(\{e_v : v\in V\})$ contains a~neighborhood of zero in $\operatorname{span} \iota(\{e_v : v\in V\})$, it will follow from~\Cref{fact:extension} that $\iota$ extends to an~onto isomorphism $\tilde \iota : \ell_p(V)\to\mathcal F_p ([0, 1]^d, |\cdot|^\alpha)$. To that end, let us first establish the inclusion $\beta\iota(\{e_v : v\in V\})\subseteq B_{\mathcal F_p}$ for some $\beta > 0$.

    For any $v \in V_0\setminus V_{-1}$, we deduce \[ \left\Vert \iota(e_v) \right\Vert_p = \left\Vert \delta(v) \right\Vert_p = |v|^\alpha \left\Vert\frac{\delta(v)}{|v|^\alpha} \right\Vert_p \leq d^\alpha \text, \] where $\left\Vert\frac{\delta(v)}{|v|^\alpha} \right\Vert_p = 1$ as $\delta$ is an isometry.
    
    Let $v =(v_i)_{i=1}^d\in V_k\setminus V_{k-1}$, where $k\in\mathbb N$. We first note that
    \begin{equation}\label{eq:isomorphism_norm}
        \begin{split}
            \left\Vert \iota(e_v) \right\Vert_p^p &= \left\Vert \sum_{u \in V_{k-1}} 2^{k\alpha}\Lambda^d_{2^{-k+1}}(u, v)(\delta(v)-\delta(u))\right\Vert_p^p\\
            &\leq \sum_{u \in V_{k-1}} \left\Vert 2^{k\alpha}|v-u|^\alpha\Lambda^d_{2^{-k+1}}(u, v)\frac{\delta(v)-\delta(u)}{|v-u|^\alpha}\right\Vert_p^p\\
            &= \sum_{u \in V_{k-1}} (2^{k\alpha}|v-u|^\alpha\Lambda^d_{2^{-k+1}}(u, v))^p \text,
        \end{split}
    \end{equation}
    where the third equality follows as $\left\Vert \frac{\delta(v)-\delta(u)}{|v-u|^\alpha} \right\Vert_p = 1$ for any $u\in V_{k-1}$, by isometry of $\delta$. We further remark that for any $u=(u_i)_{i=1}^d\in V_{k-1}$, $\Lambda^d_{2^{-k+1}}(u, v)\neq 0$, it follows from~\Cref{lemma:cf:coeff_Lambda_properties}~\labelcref{it:Lambda_4} that $|v_i-u_i|\in \{0, 2^{-k}\}$, where $i\in\{ 1, \ldots, d\}$; hence, $2^{k\alpha}|v-u|^\alpha \leq d^{\alpha}$. Similarly, we recall that $\sum_{u \in V_{k-1}} \Lambda^d_{2^{-k+1}}(u, v) =1$ by~\Cref{lemma:cf:coeff_Lambda_properties}~\labelcref{it:Lambda_3} and $\Lambda^d_{2^{-k+1}}(\cdot, v)\geq 0$ by definition. Continuing~\eqref{eq:isomorphism_norm}, we obtain
    \begin{equation*}
        \begin{split}
            \left\Vert \iota(e_v) \right\Vert_p &\leq \left(\sum_{u \in V_{k-1}} (2^{k\alpha}|v-u|^\alpha\Lambda^d_{2^{-k+1}}(u, v))^p\right)^{1/p} \\
            &\leq d^\alpha C(p, 2^d)\text.
        \end{split}
    \end{equation*}
    
    It follows that $\iota(\{e_v : v\in V\})\subseteq d^\alpha C(p, 2^d) B_{\mathcal F_p}$.

    We verify that $\operatorname{Mol} (V \cup \{ 0 \}) \subseteq \alpha\aconv_p\,\iota(\{e_v : v\in V\})$ for some $\alpha > 0$.
    
    To that end, let $v\in V_k$, where $k\in\mathbb N_0$. Whenever $v \in V_k \setminus V_{k-1}$, we have that $2^{k\alpha}\left(\delta(v)-\sum_{u \in V_{k-1}} \Lambda^d_{2^{-k+1}}(u, v)\delta(u)\right) \in \iota(\{e_{v'} :v' \in V\})$. If $v \in V_{k-1}$, we recall that $\Lambda^d_{2^{-k+1}} (u, v) = \delta_{u, v}$ for any $u \in V_{k-1}$ by~\Cref{lemma:cf:coeff_Lambda_properties}~\labelcref{it:Lambda_2}; hence, $2^{k\alpha}\left(\delta(v)-\sum_{u \in V_{k-1}} \Lambda^d_{2^{-k+1}}(u, v)\delta(u)\right)=0$. We deduce that
    \begin{align*}
        2^{k\alpha}\left(\delta(v)-\sum_{u \in V_{k-1}} \Lambda^d_{2^{-k+1}}(u, v)\delta(u)\right) & \in \{ 0, \iota(e_v)\}\\
        &\subseteq \aconv_p \iota(\{e_{v'} :v' \in V\}) \text.
    \end{align*}
    
    Since now $\aconv_p \iota(\{e_v :v\in V\})$ satisfies the assumption of~\Cref{lemma:step} and, in particular, we have that $\{ \delta(u) : u\in\{0, 1\}^d \} = \{ \iota(e_v) : v\in V_0\setminus V_{-1} \} \cup \{ 0 \} \subseteq \aconv_p \iota(\{e_v : v\in V\})$, by~\Cref{lemma:combinations} there exists $\alpha > 0$ such that \begin{equation}\label{eq:norming_image} \operatorname{Mol} (V\cup \{ 0 \}) \subseteq \alpha\aconv_p \iota(\{e_v : v\in V\}) \text.\end{equation} Quantitatively, if $\rho$ and $\tau$ are as in~\Cref{lemma:combinations,lemma:step}, respectively, we may set $\alpha = \rho^d\tau^d$.
    
    Note that $\operatorname{Mol} (V\cup \{ 0 \})$ is a~dense subset of $\operatorname{Mol} ([0,1]^d)$, which is isometrically $p$-norming by~\Cref{fact:cf:isometrically_norming}. Hence, we get from~\eqref{eq:norming_image} that $B_{\mathcal F_p} \subseteq \alpha\,\overline \aconv_p \iota(\{e_v : v\in V\})$. Since also $\iota(\{e_v : v\in V\}) \subseteq d^\alpha C(p, 2^d) B_{\mathcal F_p}$ by the already proven part, we conclude that $\iota(\{e_v : v\in V\})$ is $p$-norming in $\mathcal F_p ([0, 1]^d, |\cdot|^\alpha)$. 
    
    Moreover, \Cref{lemma:isometrically_norming_subset} in conjuction with~\eqref{eq:norming_image} shows that $\aconv_p \iota(\{e_v : v\in V\})\supseteq 1/\alpha\cdot \operatorname{Mol} (V\cup \{ 0 \})$ contains a~neighborhood of zero in $\operatorname{span} \iota(\{e_v : v\in V\}) = \mathcal P (V\cup \{ 0 \})$.

    Altogether, we have verified that the assumptions of~\Cref{fact:extension} are satisfied; hence, $\iota$ extends to an~onto isomorphism $\tilde \iota$.
    
    Quantitatively, we have shown that \[ \frac 1 {C(p, 2^d)}\, \overline\aconv_p\,\iota(\{e_v : v\in V\}) \subseteq B_{\mathcal F_p} \subseteq \rho^d\tau^d\, \overline\aconv_p\,\iota(\{e_v : v\in V\}) \text,\] and thus the Banach-Mazur distance of $\mathcal F_p ([0, 1]^d, |\cdot|^\alpha)$ and $\ell_p (V)$ may be estimated as $d(\mathcal F_p ([0, 1]^d, |\cdot|^\alpha), \ell_p (V)) \leq \lVert \tilde\iota \rVert \lVert {\tilde\iota}^{-1}\rVert \leq C(p, 2^d)\rho^d\tau^d$.
    
    The proof is complete.

\end{proof}

\subsection{The Isomorphism~\texorpdfstring{$\mathcal F_p (\mathcal M, \rho^\alpha)$}{Fp(M, rho alpha)}, Where~\texorpdfstring{$(\mathcal M, \rho)$}{(M, rho)} Is Infinite Doubling}

We generalize the isomorphism theorem to Hölder distortions of infinite doubling metric spaces. The following argument was suggested in~\cite{Albiac2021sums}.

\begin{thm}\label{thm:isomorphism_p}
    Let $(\mathcal M, \rho)$ be an~infinite doubling metric space and $0 < \alpha < 1$, $0<p\leq 1$. Then $\mathcal F_p (\mathcal M, \rho^\alpha)$ is isomorphic to the space~$\ell_p$.
\end{thm}

\begin{proof}
    Let $\beta \in (\alpha, 1)$. By Assouad's theorem, see~\cite[Proposition~2.6]{Assouad1983}, it follows that $(\mathcal M, \rho^{\alpha / \beta})$ is bi-Lipschitz equivalent to $(\mathcal M', |\cdot|)$ for some $\mathcal M' \subseteq \mathbb R^d$, where $d\in\mathbb N$. As a~consequence, $(\mathcal M, \rho^{\alpha})$ is bi-Lipschitz equivalent to $(\mathcal M', |\cdot|^\beta)$. Hence, for the rest of the proof we may assume that $\mathcal M$ is an~infinite subset of $\mathbb R^d$.

    We note that $\mathcal F_p (\mathbb R^d, |\cdot|^\alpha)$ contains a~complemented subspace isomorphic to $\mathcal F_p (\mathcal M, |\cdot|^\alpha)$, which is infinite-dimensional since $\mathcal M$ is infinite. In fact, it is true that for any $\mathcal M \subseteq \mathcal N \subseteq \mathbb R^d$, there exists a~bounded operator $T:\mathcal F_p (\mathcal N)\to \mathcal F_p (\mathcal M)$ satisfying $T\circ L_i= \operatorname{Id}_{\mathcal F_p (\mathcal M)}$, where $L_i$ is the canonical linearization of the inclusion map $i:\mathcal M \to \mathcal N$, i.e., $\mathcal M$ is~\emph{complementably $p$-amenable} in $\mathcal N$, see~\cite[Theorem~5.1]{Albiac2021sums}.

    Moreover, it is known that $\mathcal F_p (\mathbb R^d, |\cdot|^\alpha) \simeq \mathcal F_p ([0, 1]^d, |\cdot|^\alpha)$, see~\cite[Theorem~4.15]{Albiac2021sums}. In light of~\Cref{thm:isomorphism_p_cube}, it follows that  $\mathcal F_p (\mathbb R^d, |\cdot|^\alpha) \simeq \ell_p$. 
    
    We deduce that $\mathcal F_p(\mathcal M, |\cdot|^\alpha)$ is isomorphic to a~complemented infinite-dimensional subspace of $\ell_p$. Hence, it is isomorphic to $\ell_p$, see~\cite[Theorem~1]{Pelczyński1960} and~\cite[Theorem~2]{Stiles1972}, and this finishes the proof.
\end{proof}

\emergencystretch=0.5em
\printbibliography
\end{document}